\documentclass[12pt]{amsart}

\usepackage{amssymb}

\usepackage{enumitem}

\usepackage{graphicx}

\makeatletter
\@namedef{subjclassname@2020}{%
  \textup{2020} Mathematics Subject Classification}
\makeatother

\usepackage[T1]{fontenc}

\newtheorem{theorem}[equation]{Theorem}
\newtheorem{conjecture}[equation]{Conjecture}

\newtheorem{lemma}[equation]{Lemma}
\newtheorem{proposition}[equation]{Proposition}

\theoremstyle{definition}
\newtheorem{definition}[equation]{Definition}
\newtheorem{remark}[equation]{Remark}

\numberwithin{equation}{section}

\begin{document}


\baselineskip=17pt


\title{The cubic Pell Equation L-function}
\author{Dorian Goldfeld}
\address{Department of Mathematics \\ Columbia University \\ New York \\NY 10027 \\USA\\
}
\email{goldfeld@columbia.edu}

\author{Gerhardt Hinkle}
\address{Department of Mathematics \\ Columbia University \\ New York \\NY 10027 \\USA}
\email{gnh2109@columbia.edu}

\date{}
\dedicatory{Dedicated to the memory of Andrzej Schinzel}

\begin{abstract}
For $d > 1$ a cubefree rational integer, we define an $L$-function (denoted $L_d(s)$) whose coefficients are derived from the cubic theta function for $\mathbb Q\left(\sqrt{-3}\right)$. The Dirichlet series defining $L_d(s)$ converges for $\textup{Re}(s) > 1$, and its coefficients vanish except at values corresponding to integral solutions of $mx^3 - dny^3 = 1$ in $\mathbb Q\left(\sqrt{-3}\right)$, where $m$ and $n$ are squarefree. By generalizing the methods used to prove the Takhtajan-Vinogradov trace formula, we obtain the meromorphic continuation of $L_d(s)$ to $\textup{Re}(s) > \frac{1}{2}$ and prove that away from its poles, it satisfies the bound $L_d(s) \ll |s|^{\frac{7}{2}}$ and has a possible simple pole at $s = \frac{2}{3}$, possible poles at the zeros of a certain Appell hypergeometric function, with no other poles. We conjecture that the latter case does not occur, so that $L_d(s)$ has no other poles with $\textup{Re}(s) > \frac{1}{2}$ besides the possible simple pole at $s = \frac{2}{3}$.
\end{abstract}



\maketitle

\section{Introduction} 

Fix the quadratic number field $K=\mathbb Q\left(\sqrt{-3}\right)$ with ring of integers $\mathcal O_K.$  We consider the family of cubic Pell equations
$$mx^3-dny^3=1$$
where $d>1$ is a fixed cube-free rational integer, $x,y\in\mathcal O_K$ are variables, and $m,n \in \mathcal O_K$ with $m,n$  squarefree. 

Set $\lambda=\sqrt{-3}.$ In Definition \ref{Lfunction},  the Pell equation L-function
\[
  L_d(s) := \sum_{\nu \in \lambda^{-3} \mathcal O_K } \tau(\nu) \overline{\tau(1 + d\nu)} |\nu (1 + d\nu)|^{-s}, \qquad\quad (s\in\mathbb C, \;\text{\rm Re}(s) > 1),
\]
is introduced where $\tau(\nu)$ is the Fourier coefficient of the cubic theta function (see Proposition \ref{CubicThetaFunction}). It will be shown that the coefficient $\tau(\nu)\overline{\tau(1+d\nu)}$ vanishes unless 
$$\nu = mx^3, \qquad 1+d\nu = ny^3,$$ where $m,n,x,y\in\mathcal O_K$ with $m,n$ squarefree, i.e., the coefficient of the Dirichlet series for $L_d(s)$ vanishes unless it is coming from a solution of the Pell equation.

The main result of this paper is the meromorphic continuation of the expression $\mathcal F\left(s,\frac{(d+1)^2}{2d}\right)\cdot L_d(s)$ to $\text{\rm Re}(s) > \frac12$, where
$$\mathcal F(s,x) := \int\limits_0^1 \left( t^{s-\frac43} + t^{s-\frac23}  \right) \left(x\cdot t + \tfrac{(t-1)^2}{2}\right)^{-s}  dt$$
for $s \in \mathbb C$, $\textup{Re}(s) > \frac{1}{3}$, and $x > 0$ is a special case of \'Emile Picard's integral representation \cite{MR1508705} of the Appell hypergeometric function.

\begin{theorem} \label{MainTheorem} Let $d>1$ be a  cubefree rational integer. The  function $\mathcal F\left(s,\frac{(d+1)^2}{2d}\right)\cdot L_d(s)$ has meromorphic continuation to $\text{\rm Re}(s)> \frac12$ with at most a simple pole at $s = \frac23.$

Fix $\varepsilon> 0.$ In the region $\text{\rm Re}(s)> \frac12+\varepsilon$ and $|s-\frac23| > \varepsilon$ we have the  bound
$$\left|\mathcal F\left(s,\tfrac{(d+1)^2}{2d}\right)\cdot L_d(s)\right| \; \ll_{d,\varepsilon} \;  |s|^3.$$
\end{theorem}

\begin{remark} The proof of Theorem \ref{MainTheorem} is given in \S 10  while the residue of the possible pole at $s=\frac23$ is determined in STEP 8 in \S 10. Theorem  \ref{RelatingSd(s)} introduces a  modification of $L_d(s)$ (denoted $L_d^\#(s)$) and it is proved that $L_d^\#(s)$ has meromorphic continuation to $\text{\rm Re}(s)>\frac13$ with a possible pole at $s=\frac23$ and  infinitely many poles  on the line $\text{\rm Re}(s) = \frac12$ coming from Maass cusp forms including a double pole at $s = \frac 12.$ \end{remark}

With the method of steepest descent (see Lemma \ref{SteepestDescent}) we show that the Picard integral satisfies 
$$\left|\mathcal F\left(s,\tfrac{(d + 1)^2}{2d}\right)\right| \sim \; C_d\cdot |s|^{-\frac12}$$ 
for $\text{\rm Re}(s)>\frac13$ fixed and  $\big|\text{\rm Im}(s)\big| \to\infty$,  where $C_d>0$ is a fixed constant depending at most on $d$. This allows us to  obtain the meromorphic continuation and growth of  $L_d(s)$ away from poles.

\begin{theorem}
Fix $\varepsilon > 0$. Let $d > 1$ be a cubefree rational integer. The function $L_d(s)$ has meromorphic continuation to $\textup{Re}(s) > \frac{1}{2}$ with at most a simple pole at $s = \frac{2}{3}$ and possible poles at the zeros of $\mathcal F\left(s,\frac{(d + 1)^2}{2d}\right)$ with $\frac{1}{2} < \textup{Re}(s) \leq 1$.
 In the region $\textup{Re}(s) > \frac{1}{2} + \varepsilon$  and $|s - \rho| > \varepsilon$ (for any pole $\rho$ of $L_d(s)$)  we have the bound
    $
      L_d(s) \ll_{d,\varepsilon} |s|^{\frac{7}{2}}.
   $
\end{theorem}

Numerical computations suggest that $\mathcal F\left(s,\frac{(d + 1)^2}{2d}\right)$ is nonvanishing on the relevant region, so we formulate the following conjecture.

\begin{conjecture}
Fix $\varepsilon > 0$. Let $d > 1$ be a cubefree rational integer. The function $L_d(s)$ has meromorphic continuation to $\textup{Re}(s) > \frac{1}{2}$ with at most a simple pole at $s = \frac{2}{3}$.
In the region $\textup{Re}(s) > \frac{1}{2} + \varepsilon$ and $\left|s - \frac{2}{3}\right| > \varepsilon$ we have the bound
    $
      L_d(s) \ll_{d,\varepsilon} |s|^{\frac{7}{2}}.
   $
 
\end{conjecture}

In the recent paper of  Hoffstein-Jung-Lee \cite{https://doi.org/10.48550/arxiv.2109.10434} it is proved that the L-function
$$ \sum_{\nu \in \lambda^{-3} \mathcal O_K } |\tau(\nu)|^2 |\nu|^{-2s} = 2\cdot 3^{5+3s} \frac{\left(1+3^{1-2s}   \right)\left(1-3^{-s} \right)\zeta_K(3s-1) \zeta_K(s)}{\left(1-3^{-2s}\right) \zeta_K(2s)}$$
where $\zeta_K(s)$ is the zeta function of $K=\mathbb Q(\sqrt{-3})$. This result is obtained by considering the inner product of an Eisenstein series with the square of the absolute value of the cubic theta function. The proof of Theorem \ref{MainTheorem} follows a similar approach except, like the Takhtajan-Vinogradov trace formula \cite{MR581605}, we use a Poincar\'e series instead of the Eisenstein series, and unlike both of those, we replace $|\theta|^2$ with $\theta\overline{\theta_d}$, where $\theta_d(z) = \theta(dz)$. That is, we consider the inner product $\big\langle P_1(*,s),\theta\overline{\theta_d} \big\rangle$, where $P_1$ is a Poincar\'e series, $\theta$ is the cubic theta function for $K$, and $\theta_d(z) = \theta(dz)$.

The spectral side of the trace formula is presented in \S \ref{SpectralSideSection} and is evaluated by standard methods, which gives the meromorphic continuation of the inner product of the Poincar\'e series with $\theta\overline{\theta_d}$ as well as the explicit computation of the poles and their residues. The possible pole at $s=\frac23$ comes from the continuous spectrum (see Theorem \ref{Theorem:EisensteinBound}). The poles at the eigenvalues of the Laplacian and the double pole at $s =\frac12$ come from the discrete spectrum (see Theorem \ref{SpectralSideBound}). The growth of the inner product away from the poles is obtained in \S \ref{BoundingSpectralSideSection}.

The main difficulty in proving Theorem \ref{MainTheorem} comes from the geometric side of the trace formula in \S \ref{GeometricSideSection}, which involves the function $S_d(s)$ defined in Theorem \ref{Sd(s)Theorem}. The function $S_d(s)$ is essentially the inner product under consideration with a single term involving multiple gamma functions removed; it is $S_d(s)$ that eventually gives rise to the $L$-functions, so that the spectral analysis of the inner product and knowledge of the gamma function term yields the results for those $L$-functions. Although the coefficients $\tau(\nu) \overline{\tau(1 + d\nu)}$ coming from the cubic theta function appear in $S_d(s)$, they are twisted by Appell hypergeometric functions, so it is not at all clear how to extract the meromorphic continuation of $L_d(s)$ from $S_d(s)$. What arises naturally is the Dirichlet series $L_d^\#(s)$ (defined in \S \ref{cubicPellEqLfuns}), which can be thought of as a  version of $L_d(s)$ twisted by Appell hypergeometric functions.
 
The key idea for extracting $L_d(s)$ by itself is to first express $S_d(s)$ as an integral involving a ratio of gamma functions times the Appell hypergeometric function (with the Appell hypergeometric function coming from an integral involving the product of two Bessel functions that appears when directly taking the inner product) and then shift the line of integration, picking up residues at the poles of the gamma functions. We then replace the Appell hypergeometric function appearing in the shifted integral and the residues with the \'Emile Picard integral, which enormously simplifies all subsequent computations. In particular, we obtain an integral of the form $\mathcal F(s,x)$, allowing us to use the binomial expansion of $\mathcal F(s,x)$ about $x = x_0$. This ultimately yields the same expression in the main term for each of the summands, which we can thus pull out of the sum to obtain the $\mathcal F\left(s,\frac{(d + 1)^2}{2d}\right)$ term in $L_d^{\#}(s)$. Specifically, the $k = 0$ and $k = 1$ terms of the binomial expansion yield expressions involving $L_d^{\#}(s)$ and $L_d^{\#}\left(s + \frac{1}{2}\right)$, while the later terms are much smaller.

The complete proof of the meromorphic continuation of $L_d(s)$ and $L_d^\#(s)$ as well as their growth properties is presented in 8 separate STEPS in \S \ref{RelatingSdSSection}.

We believe that the techniques used in this paper could be applied to higher-degree theta functions. However, those cases have additional issues to be dealt with that do not arise in the cubic case.

\section{Basic notation}

The following notation will be used consistently throughout this paper.

\vskip 10pt
\hskip 20pt$\bullet\;d \ne 1$ is a cubefree positive rational integer;
 
\vskip 10pt
\hskip 20pt$\bullet \;K := \mathbb Q\left(\lambda\right),$ with $\lambda=\sqrt{-3};$

\vskip 5pt
\hskip 20pt$\bullet \; \mathcal O_K := \mathbb Z\left[ e^{\frac{2\pi i}{3}}  \right]$;

\vskip 5pt
\hskip 20pt {$\bullet \; $ {\large$\left( \frac{a}{b}  \right)_{\scriptscriptstyle 3}$} (with $a,b\in \mathcal O_K$) is the cubic residue symbol for $K$.

\begin{definition} {\bf (Upper half-plane $\mathfrak h^2$)} We define
the classical upper half-plane $\mathfrak h^2 := \{x+iy \mid x\in\mathbb R, \, y>0\}.$ 
\end{definition}

\begin{definition} {\bf (Quaternionic upper half-space $\mathfrak h^3$)} The quaternions are expressions of the form $a+bi+cj+dk$ where $a,b,c,d\in \mathbb R$ and $i^2=j^2=k^2=ijk = -1.$ Further, $\mathbb C$ is identified with the set of quaternions with $j=k=0.$ We define the quaternionic upper half-space to be
$\mathfrak h^3 := \{x+jy \mid x\in\mathbb C, \, y>0\}.$ 
\end{definition}

\begin{definition}{\bf (Trace and exponential function on   $\mathfrak h^3$)} Let $z=x+jy\in\mathfrak h^3.$  Define the trace
$\text{\rm tr}(z) := 2\text{\rm Re}(x) + 2iy$
and the exponential function
$$e(z) := e^{2\pi i \text{\rm tr}(z)} = e^{-4\pi y} e^{4\pi i \text{\rm Re}(x)}.$$
\end{definition}

\begin{definition}{\bf (Action of $\text{\bf\rm SL}(2,\mathbb C)$ on $\mathfrak h^3$)} The $\text{\rm SL}(2,\mathbb C)$  action on $\mathfrak h^3$ is given by
$$\begin{pmatrix} \alpha&\beta\\\gamma&\delta \end{pmatrix}z = (\alpha z+\beta) (\gamma z+\delta)^{-1}, \qquad \big(z\in \mathfrak h^3, \;\, \alpha,\beta,\gamma,\delta\in\mathbb C, \;\alpha\delta-\beta\gamma=1\big).$$
\end{definition}

\begin{definition} {\bf (The congruence subgroup $\Gamma(N)$)}  Let 
$N\in \mathcal O_K, N\ne 0$.  We  define the congruence subgroup 
$$\Gamma(N) := \left\{\gamma\in \text{\rm SL}(2, \mathcal O_K) \; \bigg| \;\gamma \equiv \begin{pmatrix} 1&0\\0&1\end{pmatrix} \big(\hskip-12pt\mod  N\big)\right\}$$ 
and
$$\Gamma_\infty(N) := \left\{  \left(\begin{matrix}\alpha&\beta\\\gamma&\delta \end{matrix}\right)\in \text{\rm SL}(2, \mathcal O_K) \; \bigg| \; \gamma\equiv  \left(\begin{matrix}1&0\\0&1 \end{matrix}\right) (\text{\rm mod}\; N)  \right\}.$$
\end{definition}

\section{Cubic Gauss sums}

\begin{definition}{\bf (Cubic Gauss sum)}
  Fix $\lambda=\sqrt{-3}$. For any $\mu \in \lambda^{-3} \mathcal O_K$ and $a\in\mathcal O_K$ satisfying $a \equiv 1 \pmod{3}$,  the cubic Gauss sum is
$$
  g(\mu,a) := \sum_{\beta \,\in\, \mathcal O_K \hskip-7pt\pmod{a}} \left(\frac{3\beta}{a}\right)_3 e\left(\frac{3\mu\beta}{a}\right).
$$
For $\mu \in \mathcal O_K$ this simplifies to
$$
  g(\mu,a) = \sum_{\beta \,\in\, \mathcal O_K\hskip-7pt\pmod{a}} \left(\frac{\beta}{a}\right)_3 e\left(\frac{\mu\beta}{a}\right).
$$
\end{definition}

\begin{definition} {\bf (The function $\tau$)} \label{tau}
We define $\tau$ according to the following formulae, where in all of the following expressions, $a, b \in \mathcal O_K$  with $a,b \equiv 1 \pmod 3$, and $a$ is squarefree:
\[
  \tau(\mu) := \begin{cases}
    \overline{g(\lambda^2,a)} \left|\frac{b}{a}\right| 3^{\frac{n}{2} + 2} & \text{if}\;\,  \mu = \pm \lambda^{3n - 4} ab^3, \; n \in \mathbb Z_{\geq 1}, \\
    &\\
    e^{-\frac{2\pi i}{9}} \overline{g(\omega\lambda^2,a)} \left|\frac{b}{a}\right| 3^{\frac{n}{2} + 2} & \text{if}\;\, \mu = \pm\omega\lambda^{3n - 4} ab^3, \; n \in\mathbb Z_{\geq 1}, \\
    &\\
    e^{\frac{2\pi i}{9}} \overline{g(\omega^2\lambda^2,a)} \left|\frac{b}{a}\right| 3^{\frac{n}{2} + 2} & \text{if}\;\, \mu = \pm\omega^2\lambda^{3n - 4} ab^3, \; n \in \mathbb Z_{\geq 1}, \\
    &\\
    \overline{g(1,a)} \left|\frac{b}{a}\right| 3^{\frac{n + 5}{2}} & \text{if}\;\, \mu = \pm\lambda^{3n - 3} ab^3, \; n \in \mathbb Z_{\geq 0}, \\
    &\\
    0 & \text{otherwise}.
  \end{cases}
\]
\end{definition}

\begin{remark}
  From unique factorization in $K$, any $\mu \in \lambda^{-3} \mathcal O_K$ has at most one of the above forms, and if it does have one of those forms, the values of $a$, $b$, and $n$ are unique, so the above definition is well-defined and we may denote by $a(\mu)$, $b(\mu)$, and $n(\mu)$ the values of $a$, $b$, and $n$ in the decomposition of $\mu$.
\end{remark}

\section{Cubic theta function}

Patterson \cite{MR563068}, following Kubota \cite{MR0255490}, gave a detailed explicit study of the simplest cubic theta function for the field $\mathbb Q(\sqrt{-3})$ which we briefly review. Let $\Gamma$ be any congruence subgroup of SL$(2,\mathcal O_K)$ whose level is divisible by 9. Then for $\gamma = \left(\begin{matrix}a&b\\c&d \end{matrix}\right) \in \Gamma$ define
$$\kappa(\gamma) := \begin{cases} \left(\frac{c}{d}\right)_3 & \text{if}\; c\ne 0,\\
 \;\;1 &\text{if}\; c=0, \end{cases}$$
 to be the Kubota symbol on $\Gamma$. This allows us to construct a metaplectic Eisenstein series.
 \begin{definition} {\bf (Cubic metaplectic Eisenstein series for $\mathbb Q(\sqrt{-3})$)} Let 
$s\in\mathbb C$ with $\text{\rm Re}(s) > 1.$ For $z=x+jy\in\mathfrak h^3$  and $I_s(z) := y^s$, we define the cubic metaplectic Eisenstein series by
 $$E^{(3)}(z,s) := \sum_{\gamma \in\Gamma_\infty(9)\backslash\Gamma(9)} \overline{\kappa(\gamma)}\, I_s(\gamma z)^{2s}.$$ 
 \end{definition}
 
 The Eisenstein series $E^{(3)}$ satisfies the automorphic relation
 $$E^{(3)}(s,\gamma z) = \kappa(\gamma) E^{(3)}(s, z)$$
 for all $\gamma \in \Gamma(9)$ and has a simple pole at $s=\frac23$ with residue the cubic theta function defined by
 $$\theta(z) = 2\,\underset{s=\frac23}{\text{\rm Res}}\; E^{(3)}(z,s).$$
 
\begin{proposition}{\bf (Fourier expansion of $\theta(z)$)}
\label{CubicThetaFunction} Let $z=x+jy \in\mathfrak h^3.$ The Fourier expansion of the cubic theta function for $\mathbb Q(\sqrt{-3})$ is given by
\begin{align*}
  \theta(z) := \sigma y^{\frac{2}{3}}\, + \underset{\mu\ne 0}{\sum_{ \mu\, \in\, \lambda^{-3}\cdot \mathcal O_K}} \tau(\mu) y K_{\frac{1}{3}}(4\pi|\mu|y) e(\mu z),
\end{align*}
where $\sigma = \frac{9\sqrt{3}}{2}$ and $K_{\frac{1}{3}}$ is the modified Bessel function of the second kind with order $\frac{1}{3}$.
\end{proposition}

\begin{proof} This was proved by Patterson \cite{MR563068}, \cite{MR563069}.
\end{proof}

 \vskip 15pt
\section{The cubic Pell equation L-functions $ L_d(s),  L_d^*(s),  L_d^\#(s)$} \label{cubicPellEqLfuns}

\begin{definition} {\bf (The cubic Pell equation $L$-function $L_d(s)$)} 
\label{Lfunction}
Let $d\ne1$ be a cubefree positive integer. Then for $s\in\mathbb C$ with $\text{\rm Re}(s)$ sufficiently large we define the $L$-function
\[
  L_d(s) := \sum_{\nu \in \lambda^{-3} \mathcal O_K } \tau(\nu) \overline{\tau(1 + d\nu)} |\nu (1 + d\nu)|^{-s}.
\]
\end{definition}

\begin{proposition} \label{Ld(s)convergence}
  The Dirichlet series $L_d(s)$ in  Definition \ref{Lfunction}  converges absolutely  for $\text{\rm Re}(s) > 1$.
\end{proposition}

\begin{proof}
  Each of the Gauss sums that may occur in the definition of $\tau$ has absolute value less than or equal to $|a|$ (see \cite{MR1041052}), so
  \[
    |\tau(\mu)| \ll 3^{\frac{n}{2}} |b|.
  \]
  for $u$ of the general form
$\mu=\pm \omega^j \lambda^{3n-k} a b^3 \; (0\ne a,b \in \mathcal O_K, \; n\in\mathbb Z_{\ge0}   )$ with
$0\le j\le 2$ and $k\in\{3,4\}$.
  By Cauchy's inequality, we have
\begin{align*}\sum_\nu  \frac{\tau(\nu) 
\overline{\tau(1+d\nu)}}{|\nu(1+d\nu)|^{1+\varepsilon}} 
& \le
\left(\sum_\nu \frac{|\tau(\nu)|^2}{|\nu|^{2+2\varepsilon}}\right)^{\frac12}\cdot \left(\sum_\nu \frac{|\tau(1+d\nu)|^2}{|(1+d\nu)|^{2+2\varepsilon}}\right)^{\frac12}\\
& \ll \left(\underset{a\ne 0}{\sum_{a\in\mathcal O_K }}\;
\underset{b\ne 0}{\sum_{b\in\mathcal O_K }}\;\sum_{n=1}^\infty \frac{3^{n}\, |b|^2}{\left(|a|^2\, 3^{3n}\,|b|^6\right)^{1+\varepsilon}}  \right)
\end{align*}

Since the Dirichlet series
$$\underset{a\ne 0}{\sum_{a\in\mathcal O_K }} |a|^{-s}$$
converges absolutely for $\text{Re}(s)>2$ and the sums over $n,b$ also converge for $\text{Re}(s)>2$  this completes the proof.
\end{proof}

\begin{definition} {\bf (The cubic Pell equation $L$-function $L_d^*(s)$)} 
\label{LfunctionStar}
Let $d\ne1$ be a cubefree positive integer. Then for $s\in\mathbb C$ with $\text{\rm Re}(s)$ sufficiently large we define the $L$-function
\[
  L_d^*(s) := \sum_{\nu \in \lambda^{-3} \mathcal O_K } \tau(\nu) \overline{\tau(1 + d\nu)} |\nu (1 + d\nu)|^{-s}\left(a_d(\nu)-\tfrac{d^2+1}{2d}\right)\]
  where 
  $$a_d(\nu) := \;\frac{ |\nu|^2+ |1+d\nu|^2-1}{2|\nu|\cdot|1+d\nu|} = \frac{d^2+1}{2d}\bigg(1+\mathcal O_d\left(|\nu(1+dv)|^{-\frac12}\right)\bigg).$$
  
  \begin{remark} It follows from Proposition
  \ref{Ld(s)convergence} that the Dirichlet series for $L_d^*(s)$ converges absolutely for $\text{\rm Re}(s) >\frac12.$
  \end{remark}

\end{definition}

\begin{definition} {\bf (The cubic Pell equation $L$-function $L_d^\#(s)$)} 
Let $d\ne1$ be a cubefree positive integer. Then for $s\in\mathbb C$ with $\text{\rm Re}(s)$ sufficiently large we define the $L$-function
\[
    L_d^{\#}(s) = \mathcal F\left(s,\tfrac{(d + 1)^2}{2d}\right) L_d(s) \;-\; s\cdot \mathcal F\left(s + 1,\tfrac{(d + 1)^2}{2d}\right) L_d^*(s).
  \]

\end{definition}

 \vskip 15pt
\section{Spectral decomposition of $\mathcal L^2\left(\Gamma(9d^2) \backslash \mathfrak h^3\right)$ }

Let $z = x+jy\in\mathfrak h^3$ where $x = x_0+ix_1\in\mathbb C$ with $x_1,x_2\in\mathbb R.$ The Laplace-Beltrami differential operator on $\mathfrak h^3$ is given by
$$\Delta := y^2\left(\frac{\partial^2}{\partial x_1^2} + \frac{\partial^2}{\partial x_2^2} + \frac{\partial^2}{\partial y^2}   \right) \; - \; y\frac{\partial}{\partial y}. $$
 Recall that $d\ne 1$ is a cubefree positive rational integer. The Hilbert space $\mathcal L^2\left(\Gamma(9d^2) \backslash \mathfrak h^3\right)$ decomposes into eigenfunctions of $\Delta$ given by Maass cusp forms, Eisenstein series, and residues of Eisenstein series. The Maass cusp forms $u_j(z) \; (j=1,2,3\ldots)$ 
 have Laplace-Beltrami eigenvalues $\lambda_j = 2s_j (2 - 2s_j)$, where $s_j = \frac{1}{2} + it_j$. 
  The Selberg eigenvalue conjecture predicts that all $t_j\in\mathbb R$, and it is known that there can only be finitely many $t_j\in i\cdot \mathbb R.$  The Maass cusp forms  satisfy the automorphic relation, $u_j(\gamma z) = u_j(z)$ for all $z \in \mathfrak h^3$ and $\gamma \in \Gamma\left(9d^2\right)$, and their Fourier expansions are given by
  
 \begin{equation} \label{MaassFourierExp}
  u_j(z) = \sum_{\underset{m \neq 0}{\scriptscriptstyle{m\, \in\, \lambda^{-3} \cdot \mathcal O_K}}} c_j(m)\, y\, K_{2s_j - 1}(4\pi|m|y)) e(mx), \qquad \quad \big(c_j(m)\in\mathbb C\big),
  \end{equation}
  where for $v\in\mathbb C$
  $$K_v(y) := \frac12 \int\limits_0^\infty e^{-\frac12 y\left(u+u^{-1}\right)}\; u^v \; \frac{du}{u}$$
is the modified Bessel function of the second kind.
The Maass forms are normalized so that $\langle u_j,u_j \rangle = 1$.

Let $\kappa_1 =\infty$ and $\kappa_2, \kappa_3, \ldots, \kappa_r \in\mathbb C$ denote the cusps of $\Gamma(9d^2)$.
The continuous spectrum consists of Eisenstein series $E_{\kappa_{\ell}}\; (\ell=1,2,\ldots,r)$  where the Eisenstein series corresponding to the cusp at infinity is defined by
\[
  E_{\infty}(z,s) = \sum_{\gamma\, \in \, \Gamma_{\infty}(9d^2) \backslash \Gamma(9d^2)} I_s(\gamma z)
\]
and the Eisenstein series corresponding to another cusp $\kappa_{\ell}$ is defined by
\[
  E_{\kappa_{\ell}}(z,s) = E_{\infty}(\alpha z,s),
\]
where $\alpha \kappa_{\ell} = \infty$. For a cusp $\kappa_{\ell}$, the Fourier expansion of $E_{\kappa_{\ell}}(z,s)$ is
 
\begin{equation} \label{EisensteinSeries}
  E_{\kappa_{\ell}}(z,s) = \delta_{\kappa_{\ell},\infty} y^{2s} + c_{\kappa_{\ell}}(0,s) y^{2 - 2s} + \sum_{\underset{m \neq 0}{\scriptscriptstyle{m \in \lambda^{-3} \mathcal O_K}}} \hskip-5pt c_{\kappa_{\ell}}(m,s)\, y\, K_{2s - 1}(4\pi|m|y)\cdot e(m x)
\end{equation}
where the Fourier coefficients $c_{\kappa_\ell}(m,s) \in \mathbb C.$ 
\begin{definition} {\bf (Petersson inner product)} For two functions $F,G \in \mathcal L^2\left(\Gamma(9d^2)\backslash \mathfrak h^3    \right)$ we define the inner product
$$\big\langle F, G\big\rangle \; := \int\limits_{\Gamma(9d^2)\backslash \mathfrak h^3} F(z) \, \overline{G(z)  }\;\, \frac{dx dy}{y^3}.$$
\end{definition}
\begin{proposition} {\bf (Spectral decomposition of  $\mathcal L^2\left(\Gamma(9d^2) \backslash \mathfrak h^3\right)$)} \label{SpectralDecomposition} Suppose that $u_j$ ($j=1,2,\ldots$) is an orthonormal basis of Maass cusp forms for
$\mathcal L^2\left(\Gamma(9d^2) \backslash \mathfrak h^3\right)$ and let $E_{\kappa_\ell}$ ($\ell=1,2,\ldots,r$) be the Eisenstein series in (\ref{EisensteinSeries}). Define $u_0(z)$ to be the constant function $\equiv \text{\rm Vol}\left(\Gamma(9d^2)\backslash \mathfrak h^3\right)^{-\frac12}.$

\vskip 3pt
Let $F\in \mathcal L^2\left(\Gamma(9d^2) \backslash \mathfrak h^3\right)$.  Then
\begin{align*}
F(z) = \sum_{j=0}^\infty \big\langle F, u_j\big\rangle
+\sum_{\ell=1}^r \int\limits_{-\infty}^\infty\Big\langle  F,\; E_{\kappa_\ell}\left(*,\;\tfrac12+iu \right)     \Big\rangle \cdot  E_{\kappa_\ell}\left(z,\;\tfrac12+iu \right) \,du.
\end{align*}
\end{proposition}

\begin{proof} See \cite{MR723012}.
\end{proof}
 
 \vskip 15pt
\section{Spectral Side of the  cubic Takhtajan-Vinogradov trace formula}
\label{SpectralSideSection}

We begin with the definition of the Poincar\'e series which plays a crucial role in the evaluation of the cubic Takhtajan-Vinogradov trace formula.
\begin{definition}{\bf (Poincar\'e series)} \label{Def:PoincareSeries} Let $d>1$ be a rational cubefree integer and let $n > 0$ with $n\in\mathbb Z.$  Then for $z=x+jy\in\mathfrak h^3$ and $I_s(z) = y^s$ for $s\in\mathbb C$ we define the Poincar\'e series
$$P_n(z,s) := \sum_{\gamma \,\in\, \Gamma_\infty(9d^2)\backslash\Gamma(9d^2)} I_s(\gamma z)\, e(n\gamma z)$$
which converges absolutely and uniformly on compact subsets of $s\in\mathbb C$ with $\text{\rm Re}(s)> 1.$ 
\end{definition}

We also define $\theta_d(z) := \theta (dz)$.
The cubic Takhtajan-Vinogradov trace formula is an identity that is obtained by computing the inner product 
$\Big\langle P_1(*,s), \;\theta\, \overline{\theta_d}\Big\rangle
$
in two different ways. The first way to compute the inner product is by replacing the Poincar\'e series $P_1$ with its spectral expansion into Maass cusp forms $u_j \;(j=1,2,\ldots)$ and integrals of Eisenstein series $E_{\kappa_\ell} (\ell = 1,2,\ldots r)$ where $\kappa_1,\ldots \kappa_r$ are the cusps of $\Gamma(9d^2).$ 

\begin{proposition}{\bf (Spectral decomposition of $P_1$)}  \label{spectral} Let $\text{\rm Re}(s)>1.$ Then we have the spectral expansion
  \begin{align*}
     P_1(z,s) & = \sum_{j = 1}^{\infty} \Big\langle P_1(*,s),u_j \Big\rangle \cdot u_j(z) 
                + \frac{1}{4\pi} \sum_{\ell = 1}^r \int\limits_{-\infty}^{\infty} \Big\langle P_1(*,s),E_{\kappa_{\ell}}\left(*,\tfrac12 + iu\right) \Big\rangle    \\
              &\hskip 240pt \cdot  E_{\kappa_{\ell}}\left(z,\tfrac12 + iu\right) du.
  \end{align*}
\end{proposition}
\begin{proof} This follows immediately from Proposition \ref{SpectralDecomposition}. 
\end{proof}
\begin{remark}
The Poincar\'e series is orthogonal to the residual spectrum which consists only of the constant function $u_0(z),$ so this term does not appear in the spectral expansion.
\end{remark}

\begin{theorem}\label{SpectralSide} {\bf (Spectral side of the cubic Takhtajan-Vinogradov trace formula)} Let $\text{\rm Re}(s)>1.$
  Then 
  
$$\boxed{\Big\langle P_1(*,s), \;\theta\, \overline{\theta_d}\Big\rangle
 \, = \; \mathcal C(s) + \mathcal E(s)}
$$
where
\begin{align*}
\mathcal C(s) := \sum_{j = 1}^{\infty} \Big\langle P_1(*,s), \,u_j \Big\rangle\cdot \big\langle u_j, \,\theta \overline{\theta_d} \big\rangle
\end{align*}
is the cuspidal contribution and 
$$\mathcal E(s) =  \frac{1}{4\pi} \sum_{\ell = 1}^r \int\limits_{-\infty}^{\infty} \Big\langle P_1(*,s), \,E_{\kappa_{\ell}}\left(*,\tfrac12 + iu\right) \Big\rangle    \cdot \Big\langle E_{\kappa_{\ell}}\left(*,\tfrac12 + iu\right),\;\theta \overline{\theta_d} \Big\rangle \,du$$
is the Eisenstein contribution.
\end{theorem}
\begin{proof} This follows immediately after taking the inner product of the\linebreak spectral decomposition of $P_1$ given in Proposition \ref{spectral} with $\theta \overline{\theta_d}.$
\end{proof}

\section{Bounding the spectral side}
\label{BoundingSpectralSideSection}

In this section we will obtain the meromorphic continuation and bounds (away from poles) for both the cuspidal and Eisenstein contributions to the spectral side of the cubic Takhtajan-Vinogradov trace formula given in Theorem \ref{SpectralSide}.
 
\pagebreak

 \begin{theorem}{\bf (Bound for the spectral side)}  \label{SpectralSideBound}
 The spectral contribution $\mathcal C(s) +\mathcal E(s)$ given in Theorem \ref{SpectralSide} has meromorphic continuation to $\text{\rm Re}(s)>0$ whose set of  poles $\mathcal P$ in this region are a double pole at $s=\tfrac12$ and possible simple poles at $s=\frac23$ and $ s= s_j,\;1-s_j \;(\text{for}\; j=1,2,\ldots)$ where\linebreak
  $\lambda_j=2s_j(2-2s_j) \ne 1$ is the Laplace-Beltrami eigenvalue of a Maass cusp form $u_j$. The poles at $s = s_j, \; 1 - s_j$ occur if and only if $\big\langle u_j,\theta \overline{\theta_d}\big\rangle\ne 0.$
\vskip 5pt 
Fix $\varepsilon >0$ (sufficiently small) we define the region (away from poles)
    $$\mathcal R_\varepsilon :=\Big\{s\in\mathbb C  \;\Big\vert \;\text{\rm Re}(s)>\varepsilon,\; |s-\rho|>\varepsilon \;\text{for all $\rho\in\mathcal P$}  \Big\}.$$
    For $s \in \mathcal R_{\varepsilon}$, we have the bound
    \[
      \Big\langle P_1(*,s),\;\theta\overline{\theta_d} \Big\rangle \; \ll_{\varepsilon} \; |s|^{\textup{max}\left(2\text{\rm Re}(s) + \frac{5}{6}, \;\frac{4}{3}\right) + \varepsilon} e^{-\pi|s|}.
    \]
  \end{theorem}
\begin{proof}
The existence of the double pole at $s=\tfrac12$ is proved in Proposition \ref{PoleAt1/2}. The existence of possible simple poles at the eigenvalues of the Laplacian is a consequence of the fact that
\begin{align*}
\mathcal C(s) := \sum_{j = 1}^{\infty} \Big\langle P_1(*,s), \,u_j \Big\rangle\cdot \big\langle u_j, \,\theta \overline{\theta_d} \big\rangle
\end{align*}
together with the identity (\ref{discrete}) which represents $\big\langle P_1(*,s), \,u_j \big\rangle$ in terms of Gamma factors with poles at the eigenvalues of the Laplacian. The possible simple pole at $s=\frac23$ comes from the continuous spectrum (see Theorem \ref{Theorem:EisensteinBound}).
\vskip 5pt
    In Theorems \ref{Theorem:SpectralBound}      and \ref{Theorem:EisensteinBound}, we prove the bounds
    \[
      \mathcal C(s) \ll_{\varepsilon}\, |s|^{\textup{max}\left(2\text{\rm Re}(s) + \frac{5}{6}, \;\frac{4}{3}\right) + \varepsilon} e^{-\pi|s|}
    \]
    and
    \[
        \mathcal E(s) \ll_{\varepsilon}\, |s|^{2\text{\rm Re}(s) - \frac{1}{2}} e^{-\pi|s|}.
      \]

    By applying Theorem \ref{SpectralSide} and noting that the first of these two upper bounds is always larger,  Theorem \ref{SpectralSideBound} immediately follows.
  \end{proof}
  
  \vskip 10pt
\begin{theorem}{\bf(Bounding  $\mathcal C(s)$)}
\label{Theorem:SpectralBound} The cuspidal contribution $\mathcal C(s)$ to the spectral side has meromorphic continuation to 
$\text{\rm Re}(s)>0$ whose set of  poles $\mathcal P'$ in this region are a double pole at $s=\tfrac12$ and possible simple poles at $ s= s_j,\;1-s_j$ where $\lambda_j=2s_j(2-2s_j) \ne 1$ is the Laplace-Beltrami eigenvalue of a Maass cusp form $u_j$.  The aforementiond simple poles occur if and only if $\big\langle u_j, \,\theta \overline{\theta_d}\big\rangle\ne 0.$ Let $\mathcal R_\varepsilon^{\prime}$ be defined as
$$\mathcal R'_\varepsilon :=\Big\{s\in\mathbb C  \;\Big\vert \;\text{\rm Re}(s)>\varepsilon,\; |s-\rho|>\varepsilon \;\text{for all $\rho\in\mathcal P'$}  \Big\}.$$
\vskip 3pt

Then for $s\in\mathcal R_\varepsilon^{\prime}$ we have the bound $\mathcal C(s) \,\ll_\varepsilon \, |s|^{\textup{max}\left(2\text{\rm Re}(s) + \frac{5}{6},\;\frac{4}{3}\right) + \varepsilon} e^{-\pi|s|}$.
\end{theorem}

 We first wish to show that the sum given by the cuspical contribution $\mathcal C(s)$  in Theorem \ref{SpectralSide} converges for $s\in \mathcal R_\varepsilon$ and $0<\text{\rm Re}(s)\le 1$, that is to say for $s$ away from the poles of $\mathcal C(s)$. This will require a number of preliminary propositions after which we restate Theorem \ref{Theorem:SpectralBound} as Proposition \ref{Prop:SpectralBound} and give its proof.
 \begin{proposition}{\bf (Bounding the first coefficient of a Maass  form)} \label{MaassBound} Let $u_j\,(j=1,2,\ldots)$ be an orthonormal basis of Maass cusp forms for
$\mathcal L^2\left(\Gamma(9d^2) \backslash \mathfrak h^3\right)$ with Laplace-Beltrami eigenvalue $2s_j(2-2s_j)$ and Fourier coefficients $c_j(m)$ as in (\ref{MaassFourierExp}). Then for $s_j=\frac12+it_j$ we have the bound
$$|c_j(1)| \ll \frac{|s_j|^\varepsilon}{\big|\Gamma(s_j)\, \Gamma(1-s_j)\big|^\frac12}
\ll 
\left(1+|t_j|\right)^{\varepsilon}e^{\frac{\pi}{2} |t_j|}.$$
\end{proposition}
\begin{proof}
 By applying the bound from \cite{MR4271357} and the results from \cite{MR4134245}, we have
\[
  |c_j(1)|^2 \ll \frac{1}{L(1,\text{\rm Ad}(u_j)) \left|\Gamma\left(\frac{1}{2} + it_j\right) \Gamma\left(\frac{1}{2} - it_j\right)\right|}.
\]

We also have
$
  \frac{1}{\lambda_j^{\varepsilon}} \ll L(1,\text{\rm Ad}(u_j))
$
for any fixed $\varepsilon > 0$. This was shown in \cite{MR1828743} for Maass forms over $\mathbb Q$; an analogous argument holds in this case. Stirling's bound for the Gamma function implies the proposition.
\end{proof}

\begin{proposition} {\bf (Inner product of $P_1$ with  Maass forms and\linebreak Eisenstein series)}  Let $u_j\,(j=1,2,\ldots)$ be an orthonormal basis of Maass cusp forms for
$\mathcal L^2\left(\Gamma(9d^2) \backslash \mathfrak h^3\right)$ with Fourier coefficients $c_j(m)$ as in  (\ref{MaassFourierExp}).  Let $E_{\kappa_\ell}\, (\ell=1,2,\ldots,r)$ be the Eisenstein series in (\ref{EisensteinSeries}) with the Fourier coefficients $c_{\kappa_\ell}(m,*)$.
  We have the inner products
  \begin{align}\label{discrete}
    &\big\langle P_1(*,s),u_j \big\rangle = \frac{\overline{c_j(1)}\cdot \textup{Vol}\big(\mathbb C / \left(9d^2\mathcal O_K\right)\big)}{2^{6s-3}\, \pi^{2s-\frac32}}  \cdot \frac{\Gamma(2s - 1 + 2it_j) \Gamma(2s - 1 - 2it_j)}{\Gamma\left(2s - \frac{1}{2}\right)}\\
    &
   \nonumber \\
    &
  \label{continuous}
    \Big\langle P_1(*,s),E_{\kappa_{\ell}}\left(*,\tfrac{1}{2} + iu\right) \Big\rangle = \frac{\overline{c_{\kappa_{\ell}}\left(1,\tfrac{1}{2} + iu\right)}\cdot \textup{Vol}\big(\mathbb C / \left(9d^2\mathcal O_K\right)\big)}{2^{6s-3}\, \pi^{2s-\frac32}    }  \\ \nonumber &\hskip 217pt \cdot \frac{\Gamma(2s - 1 + 2iu) \Gamma(2s - 1 - 2iu)}{\Gamma\left(2s - \frac{1}{2}\right)}.
  \end{align}
\end{proposition}
\begin{proof}
  We explicitly write out the computation for the first of the above inner products; the second one is done by an analogous argument.
  \begin{align*}
   & \big\langle P_1(*,s),u_j \big\rangle   = \int\limits_{\Gamma(9d^2) \backslash \mathfrak h^3} \sum_{\gamma \,\in\, \Gamma_{\infty}(9d^2) \backslash \Gamma(9d^2)}
           \hskip-10pt                      
                                  I_s(\gamma z) e(\gamma z) \hskip-5pt \\ &\hskip 150pt \cdot \sum_{\underset{m \neq 0}{\scriptscriptstyle{m\, \in\, \lambda^{-3}\cdot \mathcal O_K}}} \overline{ c_j(m) y K_{2s_j - 1}(4\pi|m|y) e(mx)} \;\frac{dxdy}{y^3} \\
                                  &
                                 \hskip 15pt =  \int\limits_{\Gamma_{\infty}(9d^2)} y^{2s} e^{-4\pi y} e^{4\pi i \text{\rm Re}(x)} \sum_{\underset{m \neq 0}{\scriptscriptstyle{m\, \in\, \lambda^{-3} \mathcal O_K}}} \overline{c_j(m)} y K_{2it_j}(4\pi|m|y) e^{-4\pi i \text{\rm Re}(mx)}\frac{dxdy}{y^3} \\
                                                                &
                                 \hskip 15pt   
         = \sum_{\underset{m \neq 0}{\scriptscriptstyle{m\, \in\, \lambda^{-3}\cdot \mathcal O_K}}} \overline{c_j(m)} \hskip-10pt\int\limits_{x\,\in\,\mathbb C / (9d^2 \mathcal O_K)}\hskip-10pt e^{4\pi i \text{\rm Re}((1 - m) x)}\; dx \cdot \int\limits_0^{\infty} y^{2s - 2} e^{-4\pi y} K_{2it_j}(4\pi|m|y)\; dy \\
                                  &
                                 \hskip 15pt  
        = \; \overline{c_j(1)} \text{\rm  Vol}\left(\mathbb C / \left(9d^2 \mathcal O_K\right)\right) \int\limits_0^{\infty} y^{2s - 2} e^{-4\pi y} K_{2it_j}(4\pi y) dy \\
                                &
                                 \hskip 15pt 
          =\;  \frac{\overline{c_j(1)} \cdot \textup{Vol}\big(\mathbb C / \left(9d^2\mathcal O_K\right)\big)}{2^{6s-3}\, \pi^{2s-\frac32}     }  \cdot \frac{\Gamma(2s - 1 + 2it_j) \Gamma(2s - 1 - 2it_j)}{\Gamma\left(2s - \frac{1}{2}\right)}.
  \end{align*}
\end{proof}

\begin{proposition} {\bf (Bounding the inner product  of $P_1$ with a Maass form)} \label{P1-Maass}
Fix $\varepsilon>0.$ Let $u_j$ be a Maass cusp form with Laplace-Beltrami eigenvalue $2s_j(2-2s_j)$ and assume  $s=\sigma+it \in\mathbb C$ with  $\sigma>0, t\in\mathbb R$ satisfies $|s-s_j|>\varepsilon.$
Then  we have the bound
\begin{align*}\big\langle P_1(*,s),\, u_j\big\rangle \;\ll_\varepsilon\;  (1 + |t|)^{-2\sigma + 1}  (1+|t_j|)^\varepsilon(1 + |t - t_j|) (1 + |t + t_j|))^{2\sigma - \frac{3}{2}} \\ \cdot e^{-\pi \left(|t - t_j| + |t + t_j| - |t| - \frac{1}{2} |t_j|\right)}.\end{align*}
\end{proposition}
\begin{proof} It follows from (\ref{discrete}) and Proposition \ref{MaassBound} that
\begin{align*}
\big\langle P_1(*,s),u_j \big\rangle & \ll |c_j(1)|\cdot  \left|\frac{\Gamma(2s - 1 + 2it_j) \Gamma(2s - 1 - 2it_j)}{\Gamma\left(2s - \frac{1}{2}\right)}\right|\\
&
\ll_\varepsilon \frac{|s_j|^\varepsilon}{\big|\Gamma(s_j)\, \Gamma(1-s_j)\big|^\frac12} \cdot\left|\frac{\Gamma(2s - 1 + 2it_j) \Gamma(2s - 1 - 2it_j)}{\Gamma\left(2s - \frac{1}{2}\right)}\right|
\end{align*}
The proposition immediately follows from Stirling's bound
\begin{equation} \label{StirlingBound}
  |\Gamma(\sigma + it)| \ll (1 + |t|)^{\sigma - \frac{1}{2}} e^{-\frac{\pi}{2} |t|}.
\end{equation}
\end{proof}

\begin{lemma} {\bf (Sharp bounds for K-Bessel functions)} \label{KBesselBound} Let $t\in \mathbb R$ and $y>0.$ Then
$$
  |K_{2it}(4\pi y)| \leq  e^{-\pi |t|} f\Big(|t|,y\Big),
$$
where for $t,y >0$ the function $f(t,y)$ satisfies the following bounds.
\begin{align*} 
f(t,y) \ll  \begin{cases} \displaystyle \frac{1}{\left(t^2-4\pi^2 y^2\right)^{\frac14}}
& 
\phantom{xx}
\displaystyle \text{if}\;\; \frac{1}{4\pi} \le y \le \frac{t}{2\pi} - \frac{t^{\frac13}}{2^{\frac83}\pi},
\\
&
\\
\displaystyle \phantom{xxx} t^{-\frac13} 
&
\phantom{xx}
\displaystyle
 \text{if} \;\; \frac{t}{2\pi} - \frac{t^{\frac13}}{2^{\frac83}\pi } \le y \le \frac{t}{2\pi},
 \\ &\\
 \displaystyle
\frac{1}{\left(4\pi^2 y^2 -t^2\right)^\frac14} 
&
\phantom{xx}
\displaystyle
 \text{if}\;\; y\ge \frac{t}{2\pi}.
\end{cases}
\end{align*}
\end{lemma}

\vskip 5pt
\begin{proof} It follows from \cite{MR3044476} that
if $4\pi y \geq 2t>0$, then
\[
  f(t,y) = e^{-\sqrt{16\pi^2y^2 - 4t^2} + 2t\arccos\left(\frac{2t}{4\pi y}\right)} \min\left(\frac{\sqrt{\frac{\pi}{2}}}{\left(16\pi^2y^2 - 4t^2\right)^{\frac{1}{4}}}, \;\;\frac{\Gamma\left(\frac{1}{3}\right)}{2^{\frac{2}{3}}3^{\frac{1}{6}}} (2t)^{-\frac{1}{3}}\right).
\]

The exponent of the exponential  above is always less than zero. It follows that
$$f(t, y)  \ll \min\left( \frac{1}{\left(4\pi^2 y^2 -t^2\right)^\frac14}, \; t^{-\frac13}   \right) \leq \frac{1}{\left(4\pi^2y^2 - t^2\right)^{\frac{1}{4}}}.$$

If $1 \leq 4\pi y \leq 2t - \frac{1}{2} (2t)^{\frac{1}{3}}$, then
\[
  f(t,y) = \frac{5}{\left(4t^2 - 16\pi^2y^2\right)^{\frac{1}{4}}} \ll \frac{1}{\left(t^2-4\pi^2 y^2\right)^{\frac14}}.
\]
If $1 \leq 4\pi y < 2t$ and $4\pi y \geq 2t - \frac{1}{2} (2t)^{\frac{1}{3}}$, then
\[
  f(t,y) = 4 (2t)^{-\frac{1}{3}} \ll t^{-\frac13}.
\]
\end{proof}

\begin{proposition}{\bf(Bounding the inner product $\langle u_j, \,\theta\overline{\theta_d}\rangle$)} \label{ThetaInnerProductBound}
 Let $u_j$ be a Maass cusp form with Laplace-Beltrami eigenvalue $2s_j(2-2s_j)$ where $s_j=\tfrac12+it_j.$ Then we have the bound
 $$\big\langle u_j, \,\theta\overline{\theta_d}\big\rangle \, \ll \,  \left(1+|t_j|\right)^{-\frac{1}{6} + \varepsilon} e^{-\frac{\pi}{2}|t_j|}$$
\end{proposition}
\begin{proof}

We have
\[
  \big\langle u_j,\theta \overline{\theta_d} \big\rangle \; = \int\limits_{\Gamma(9d^2) \backslash \mathfrak h^3} u_j(z) \overline{\theta(z)} \theta(dz) \,\frac{dxdy}{y^3}.
\]
Using the Fourier expansions of $u_j$ and $\theta$ yields

\begin{align}\label{InnerProduct}
   \big\langle u_j,\theta \overline{\theta_d} \big\rangle & = \int\limits_{\Gamma(9d^2) \backslash \mathfrak h^3} u_j(z) \overline{\theta(z)} \theta(dz) \,\frac{dxdy}{y^3}  
  \\\nonumber
  &= 
 \int\limits_{\Gamma(9d^2) \backslash \mathfrak h^3} 
 \left(\;\sum^{\phantom{.......}{\pmb'}}_{n \,\in\, \lambda^{-3} \cdot\mathcal O_K} c_j(n) y K_{2it_j}(4\pi|n|y) e^{4\pi i \text{\rm Re}(nx)} \right)\\
  \nonumber &
   \hskip 50pt \cdot\left(\sigma y^\frac{2}{3} +\hskip-10pt \sum^{\phantom{.......}{\pmb'}}_{\mu\, \in\, \lambda^{-3}\cdot\mathcal O_K} \hskip-8pt\overline{\tau(\mu)} y K_{\frac{1}{3}}(4\pi|\mu|y) e^{-4\pi i \text{\rm Re}(\mu x)}\right)  \hskip-5pt  
    \\\nonumber &\hskip 72pt \cdot \left(\sigma y^{\frac{2}{3}} +\hskip-10pt \sum^{\phantom{.......}{\pmb'}}_{\nu\, \in\, \lambda^{-3}\cdot\mathcal O_K}\hskip-8pt \tau(\nu) y K_{\frac{1}{3}}(4\pi|\nu|y) e^{4\pi i \text{\rm Re}(\nu dx)}\right)\, \frac{dxdy}{y^3},
\end{align}
where $\sum\limits_n^{\phantom{......}'}$ signifies that $n=0$ is excluded from the sum.

The  integral on the right side of (\ref{InnerProduct}) is an integral of a product of three infinite sums which have very rapid convergence. After interchanging the sums with the integral, the estimation of the infinite sum of integrals  will have a dominant integral term  given by
\begin{align} \label{FirstTerm}
  \sigma^2 c_j(1) \int\limits_{\Gamma(9d^2) \backslash \mathfrak h^3}\hskip-5pt y^{\frac{7}{3}}\, K_{2it_j}(4\pi y)\,  e^{4\pi i \text{\rm Re}(x)} \,\frac{dxdy}{y^3}.
\end{align}
It follows from Proposition \ref{MaassBound} that
$$
  |c_j(1)| \ll (1+|t_j|)^{\varepsilon}e^{\frac{\pi}{2} |t_j|}.
$$
We will use the above bound for $c_j(1)$  and the  bounds for the K-Bessel function given in  Lemma  \ref{KBesselBound} to obtain the following bound for the dominant term 
 (\ref{FirstTerm}) of the inner product $\big\langle u_j,\theta \overline{\theta_d} \big\rangle:$
  \begin{equation}
  \Bigg|\sigma^2 c_j(1) \hskip-5pt\int\limits_{\Gamma(9d^2) \backslash \mathfrak h^3} K_{2it_j}(4\pi y) y^{\frac{7}{3}} e^{4\pi i \text{\rm Re}(x)} \,\frac{dxdy}{y^3}\Bigg| \; \ll \;
   (1+|t_j|)^{-\frac{1}{6} + \varepsilon} \cdot e^{-\frac{\pi}{2}|t_j|}.
 \end{equation}

Following Sarnak \cite{MR723012}, we first assume that there is only one cusp; if there are multiple cusps, we evaluate the sum over the cusps, with the computation at each cusp being the same. We evaluate the integral over the Siegel set, which in this case is $F_{\mathcal O_K} \times (a_d,\infty)$, where $F_{\mathcal O_K}$ is a fundamental domain for $\mathcal O_K$ and $a_d$ is a positive constant. Thus the absolute value of the dominant integral term (\ref{FirstTerm}) is less than a constant times
$$
(1+| t_j|)^{\varepsilon} e^{-\frac{\pi}{2} |t_j|} \int\limits_{a_d}^{\infty} f\big(|t_j|,y\big) y^{-\frac{2}{3}} \,dy \; :=\; \mathcal I_1+\mathcal I_2+\mathcal I_3,$$
   where
\begin{align*}
 \mathcal I_1 & = (1+| t_j|)^{\varepsilon} e^{-\frac{\pi}{2} |t_j|} \hskip-15pt \int\limits_{a_d}^{\frac{|t_j|}{2\pi} - \frac{|t_j|^{\frac13}}{\pi\cdot 2^{2/3}}}
 \hskip-10ptf\big(|t_j|,y\big) y^{-\frac{2}{3}} \, dy\;\\
  &\ll\; (1+|t_j|)^{-\frac{1}{6} + \varepsilon} e^{-\frac{\pi}{2}|t_j|} 
  \hskip-10pt\int\limits_{\frac{2\pi a_d}{|t_j|}}^{1 - \frac{|t_j|^{\frac23}}{2^{5/3}}}
  \hskip-7pt
   \frac{1}{u^{\frac{2}{3}} \left(1 - u^2\right)^{\frac{1}{4}}} \,du\\
  &
   \ll\;(1+|t_j|)^{-\frac{1}{6} + \varepsilon} \cdot e^{-\frac{\pi}{2}|t_j|}
\end{align*}
and
\begin{align*}
\mathcal I_2 & =   (1+| t_j|)^{\varepsilon} e^{-\frac{\pi}{2} |t_j|} \hskip-10pt \int\limits_{\frac{|t_j|}{2\pi} - \frac{|t_j|^{\frac13}}{\pi\cdot 2^{8/3} }}^{\frac{|t_j|}{2\pi}}
\hskip-9pt
 f(|t_j|,y) y^{-\frac{2}{3}} \, dy \;\\
&\ll \;
(1+|t_j|)^{-\frac{1}{3} + \varepsilon} e^{-\frac{\pi}{2}|t_j|}
\hskip-10pt
\int\limits_{\frac{|t_j|}{2\pi} - \frac{|t_j|^{\frac{1}{3}}}{\pi\cdot 2^{8/3}}}^{\frac{|t_j|}{2\pi}}
\hskip-10pt 
 y^{-\frac{2}{3}} \,dy\\
 &
   \ll \; \left(1+|t_j|\right)^{-\frac{2}{3} + \varepsilon}
e^{-\frac{\pi}{2}|t_j|},\end{align*}
while
\begin{align*}
\mathcal I_3 &= (1+|t_j|)^{\varepsilon}e^{-\frac{\pi}{2} |t_j|} \int\limits_{\frac{|t_j|}{2\pi}}^\infty f(|t_j|,y) y^{-\frac{2}{3}} \, dy\; \\
 &\ll
  \;(1+|t_j|)^{-\frac{1}{6} + \varepsilon} e^{-\frac{\pi}{2}|t_j|} \int\limits_1^{\infty} \frac{1}{u^{\frac{2}{3}} \left(u^2 - 1\right)^{\frac{1}{4}}} \,du\\
  &
    \ll \; \left(1+|t_j|\right)^{-\frac{1}{6} + \varepsilon} e^{-\frac{\pi}{2}|t_j|}.
\end{align*}

The dominant integral term (\ref{FirstTerm}) in the expression for the inner product $\big\langle u_j,\theta \overline{\theta_d} \big\rangle$  has the largest value.
 Any other integral term in the infinite sum of integrals on the right side of (\ref{InnerProduct}) includes at least one $K$-Bessel function of the form $K_{\frac{1}{3}}(4\pi|\mu|y)$, where $0\ne \mu \in \lambda^{-3}\mathcal O_K$. Because $4\pi|\mu| > 1$, the Bessel function $K_{\frac{1}{3}}(4\pi|\mu|y)$ decays exponentially as $y \rightarrow \infty$. Thus for all integral terms other than (\ref{FirstTerm}), the integrand is multiplied by an exponentially-decaying function, so its integral is much smaller than the dominant integral term (\ref{FirstTerm}).
This completes the proof of Proposition \ref{ThetaInnerProductBound}.
\end{proof}

\begin{proposition} \label{Prop:SpectralBound} Let $u_j\,(j=1,2,\ldots)$ be an orthonormal basis of Maass cusp forms for
$\mathcal L^2\left(\Gamma(9d^2) \backslash \mathfrak h^3\right)$ with Laplace-Beltrami eigenvalue $2s_j(2-2s_j)$ where $s_j=\frac12+it_j.$
  Fix $\varepsilon >0$ and define the region (away from poles)
$$\mathcal R_\varepsilon^{\prime} =\Big\{s\in\mathbb C  \;\Big\vert \;\text{\rm Re}(s)>\varepsilon, \; |s-s_j| >\varepsilon, \;  |s-(1-s_j)| >\varepsilon,\; (j=1,2,3,\ldots)\Big\}.$$
Then for $s \in \mathcal R_\varepsilon^{\prime}$ we have the bound 
  \[
    \sum_j \left|\big\langle P_1(*,s),u_j \big\rangle \big\langle u_j, \,\theta \overline{\theta_d} \big\rangle\right| 
    \ll \,
 |s|^{\textup{max}\left(2\text{\rm Re}(s) + \frac{5}{6},\;\frac{4}{3}\right) + \varepsilon} e^{-\pi|s|}.
  \]
\end{proposition}

\begin{proof} Let $s=\sigma+it$. It follows from Proposition \ref{P1-Maass} together with Stirling's bound for the Gamma function that
\begin{align*}\big\langle P_1(*,s),\, u_j\big\rangle \;\ll_\varepsilon\;  (1 + |t|)^{-2\sigma + 1}  (1+|t_j|)^\varepsilon\Big((1 + |t - t_j|) (1 + |t + t_j|)\Big)^{2\sigma - \frac{3}{2}} \\ \cdot e^{-\pi \left(|t - t_j| + |t + t_j| - |t| - \frac{1}{2} |t_j|\right)}.\end{align*}
and from Proposition \ref{ThetaInnerProductBound} that
$$\big\langle u_j, \,\theta\overline{\theta_d}\big\rangle \, \ll \, \left(1+|t_j|\right)^{-\frac{1}{6} + \varepsilon} e^{-\frac{\pi}{2}|t_j|}.$$

Combining these bounds gives
\begin{align*}
\left|\big\langle P_1(*,s),u_j \rangle \left\langle u_j, \,\theta \overline{\theta_d} \right\rangle\right| 
&\ll_\varepsilon (1 + |t|)^{-2\sigma + 1}   (1+|t_j|)^{-\frac16+\varepsilon}   \\
&
\hskip-60pt
\cdot \Big((1 + |t - t_j|) (1 + |t + t_j|)\Big)^{2\sigma - \frac{3}{2}} e^{-\pi \big(|t - t_j| + |t + t_j| - |t| \big)}.
\end{align*}

The following lemma counts the number of Laplace Beltrami eigenvalues $2s_j(2-2s_j)$ with $s_j=\frac12+it_j$ associated to Maass cusp forms
$u_j \;(j=1,2,\ldots)$ for
$\mathcal L^2\left(\Gamma(9d^2) \backslash \mathfrak h^3\right)$ in a given interval.

\begin{lemma}{\bf (Asymptotic formula for Laplace-Beltrami eigenvalues)} \label{Lemma-Laplace-Beltrami} The number of $0<t_j <t$ is  asymptotic to a constant times $t^3$ for $t\to\infty$
and the number the number of
$t< t_j < t+k\log (2+t)$ is bounded by $\ll k\cdot t^2  \log (2+t) $ for $k\ge 1$ and $t\to\infty$.
\end{lemma}
\begin{proof} See \cite{MR2261095}.
\end{proof}

We suppose that  $t = \text{\rm Im}(s)> 0$; the computations for  $t < 0$ are analogous. We  split the sum
$$\sum_j \Big|\big\langle P_1(*,s),u_j \big\rangle \big\langle u_j, \,\theta \overline{\theta_d} \big\rangle\Big |$$
  into the cases $t_j > t$ and $0 \leq t_j < t$. (We may assume that $t_j \geq 0$ because $t_j$ and $-t_j$ are completely interchangeable.) Note that the former sum has infinitely many terms while the latter sum has only finitely many terms.
  
   Let $\mathfrak i_{t,k}$ denote the interval
 $ \Big[t + k\log(2+t),\; t+(k+1)\log(2+t)\Big].$ First suppose that $t_j > t$.
 It follows from Lemma \ref{Lemma-Laplace-Beltrami} that we have
 \begin{align*}
 & \sum_{t<t_j}\; (1 + |t|)^{-2\sigma + 1}   (1+|t_j|)^{-\frac16+\varepsilon}  \Big((1 + |t - t_j|) (1 + |t + t_j|)\Big)^{2\sigma - \frac{3}{2}} 
\\ &\hskip 250pt \cdot e^{-\pi \big(|t - t_j| + |t + t_j| - |t| \big)}\\
&
\hskip 30pt
= \sum_{t<t_j} \;(1 + t)^{-2\sigma + 1}   (1+t_j)^{-\frac16+\varepsilon} \Big( (1 + t_j-t) (1 + t + t_j)\Big)^{2\sigma - \frac{3}{2}} 
\\ &\hskip 300pt \cdot e^{-\pi\left(2t_j-t\right)}\\
&
\hskip 30pt
= \sum_{k=0}^\infty \;\sum_{t_j\in \mathfrak i_{t,k}}   \;(1 + t)^{-2\sigma + 1}(1+t_j)^{-\frac16+\varepsilon} \Big( (1 + t_j-t) (1 + t + t_j)\Big)^{2\sigma - \frac{3}{2}} 
\\ &\hskip 320pt \cdot e^{-\pi\left(2t_j-t\right)}\\
&
\hskip 30pt
 \ll \sum_{k=0}^\infty \;(1+k) t^2  \log(2+ t)\cdot (1 + t)^{-2\sigma + 1}\\ &\hskip 50pt \cdot \Big(1+t+(k+1)\log(2+t)\Big)^{-\frac16+\varepsilon} \Big(1+ (k+1)\log(2+t)\Big)^{2\sigma-\frac32} \\ &\hskip 100pt \cdot \Big(1+2t +(k+1)\log(2+t)\Big)^{2\sigma - \frac{3}{2}}  (2+t)^{-\pi k} \cdot e^{-\pi t}\\
 &
\hskip 30pt
\ll (1+t)^{\frac43+\varepsilon} \cdot e^{-\pi t}.
\end{align*} 
 The last estimate above comes from the term $k=0$ since the series converges rapidly and is bounded by a positive constant times the first term.

  Now suppose that $t_j < t$. Then we have
  \begin{align*}
    &\sum_{t_j < t} (1 + |t|)^{-2\sigma + 1} (1 + |t_j|)^{-\frac{1}{6} + \varepsilon} ((1 + |t - t_j|) (1 + |t + t_j|))^{2\sigma - \frac{3}{2}} \\ &\hskip 250pt \cdot e^{-\pi (|t - t_j| + |t + t_j| - |t|)} \\
    &\hskip 50pt = \sum_{t_j < t} (1 + t)^{-2\sigma + 1} (1 + t_j)^{-\frac{1}{6} + \varepsilon} ((1 + t - t_j) (1 + t + t_j))^{2\sigma - \frac{3}{2}} e^{-\pi t} \\
    &\hskip 50pt < \sum_{t_j < t} (1 + t)^{-2\sigma + 1} (1 + t)^{-\frac{1}{6} + \varepsilon} ((1 + t) (1 + 2t))^{2\sigma - \frac{3}{2}} e^{-\pi t} \\
    &\hskip 50pt \ll \sum_{t_j < t} (1 + t)^{2\sigma - \frac{13}{6} + \varepsilon} e^{-\pi t}  \ll t^3 (1 + t)^{2\sigma - \frac{13}{6} + \varepsilon} e^{-\pi t} \\
    &\hskip 50pt \ll (1 + t)^{2\sigma + \frac{5}{6} + \varepsilon} e^{-\pi t},
  \end{align*}
  where the penultimate step uses the fact that the number of $j$ such that $t_j < t$ is asymptotically a constant times $t^3$ as $t \rightarrow \infty$ (see \cite{MR2261095}).

  Therefore, the overall bound is
  \[
    (1 + |t|)^{\textup{max}\left(2\sigma + \frac{5}{6},\frac{4}{3}\right) + \varepsilon} e^{-\pi|t|} \ll_\varepsilon \, |s|^{\textup{max}\left(2\text{\rm Re}(s) + \frac{5}{6},\frac{4}{3}\right) + \varepsilon} e^{-\pi|s|}.
  \]

 Finally, we show that the double pole at $s = \frac{1}{2}$ always occurs.

\begin{proposition}\label{PoleAt1/2}
  There exists a Maass form $u_j$ with $\lambda_j = 1$, so that the double pole of $\big\langle P_1(*,s),u_j \big\rangle$ at $s = \frac{1}{2}$ is guaranteed to occur.
\end{proposition}

\begin{proof}
 In \cite{MR31519}, Maass showed that there exists a Maass form with eigenvalue $\frac{1}{4}$ for the congruence subgroup $\Gamma_0\left(9d^2\right)$ of $\textup{SL}(2,\mathbb Z)$. For any such Maass form, we can do a base change to $K$ to get a Maass form $u_j$ with eigenvalue $\lambda_j = 1$ as desired. It is known that the lift exists and that the $L$-series of the lift is the product of the original $L$-series and the twist of that $L$-series by the quadratic character $\chi_{-3}$, from which we immediately conclude that the lift has the appropriate eigenvalue. See \cite{MR707244}.

  Because Maass forms are eigenfunctions of the Hecke operators, their first Fourier coefficient is nonzero. Thus from the inner product computed above it follows that any Maass form with eigenvalue 1 will provide a nonzero double pole contribution at $s = \frac{1}{2}$.
\end{proof}

This completes the proof of the bound for $\mathcal C(s)$.
\end{proof}

  \begin{theorem}{\bf(Bounding  $\mathcal E(s)$)}
    \label{Theorem:EisensteinBound} The Eisenstein contribution $\mathcal E(s)$ to the spectral side is meromorphic on $\textup{Re}(s) > 0$, with a possible simple pole at $s = \frac{2}{3}$ and no other poles, and for any $\varepsilon > 0$, for $\textup{Re}(s) > \varepsilon$ and $\left|s - \frac{2}{3}\right| > \varepsilon$ it satisfies the bound $\mathcal E(s) \ll_{\varepsilon} |s|^{2\text{\rm Re}(s) - \frac{1}{2}} e^{-\pi|s|}$. The residue of $\mathcal E(s)$ at $s = \frac{2}{3}$ is
      \[
        \textup{Vol}\left(\mathbb C / 9d^2\mathcal O_K\right) \sigma^2\frac{\Gamma\left(\frac{2}{3}\right)}{\zeta_K^*\left(\frac{4}{3}\right)} \left(-\overline{\tilde{c}_{\infty}}\left(1,\tfrac{2}{3}\right) + \sum_{\ell = 1}^r \overline{\tilde{c}_{\kappa_{\ell}}}\left(1,\tfrac{1}{3}\right) \tilde{c}_{\kappa_{\ell}}\left(0,\tfrac{2}{3}\right)\right).
      \]
\end{theorem}

\begin{proof}
Recall from Theorem \ref{SpectralSide} and (\ref{continuous}) that for $\text{\rm Re}(s) >1$
\begin{align*}
&\mathcal E(s) =  \frac{1}{4\pi} \sum_{\ell = 1}^r \int\limits_{-\infty}^{\infty} \Big\langle P_1(*,s), \,E_{\kappa_{\ell}}\left(*,\tfrac12 + iu\right) \Big\rangle    \cdot \Big\langle E_{\kappa_{\ell}}\left(*,\tfrac12 + iu\right),\;\theta \overline{\theta_d} \Big\rangle \,du\\
&
= 9\sqrt{3} \pi^{\frac{1}{2}}\sum_{\ell = 1}^r\int\limits_{-\infty}^\infty \overline{c_{\kappa_{\ell}}\left(1,\tfrac{1}{2} + iu\right)} \frac{\Gamma(2s - 1 + 2iu) \Gamma(2s - 1 - 2iu)}{\Gamma\left(2s - \frac{1}{2}\right)}
 \\ & \hskip 230pt \cdot \Big\langle E_{\kappa_{\ell}}\left(*,\tfrac12 + iu\right),\;\theta \overline{\theta_d} \Big\rangle\, du.
\end{align*}
\vskip 5pt
    We first bound
    \[
       \int\limits_{-\infty}^{\infty} \overline{c_{\kappa_{\ell}}\left(1,\tfrac{1}{2} + iu\right)} \Gamma(2s - 1 + 2iu) \Gamma(2s - 1 - 2iu) \Big\langle E_{\kappa_{\ell}}\left(*,\tfrac{1}{2} + iu\right),\;\theta\overline{\theta_d} \Big\rangle\,  du.
    \]
     To do this, we use an argument similar to the one used in the proof of Lemma 2.4 in \cite{https://doi.org/10.48550/arxiv.2109.10434}, though in our case we must use additional residue terms and analytic continuation.

     Define
     \begin{align*}
       &\mathcal E_{\kappa_{\ell}}(z;\mu,s) \\ &\hskip 10pt = \frac{1}{2\pi} \int\limits_{-\infty}^{\infty} \overline{c_{\kappa_{\ell}}\left(\mu,\tfrac{1}{2} + iu\right)} \Gamma(2s - 1 + 2iu) \Gamma(2s - 1 - 2iu) E_{\kappa_{\ell}}\left(z,\tfrac{1}{2} + iu\right) du.
     \end{align*}
     We know that
     \[
       c_{\kappa_{\ell}}\left(\mu,\tfrac{1}{2} + iu\right) = \frac{\tilde{c}_{\kappa_{\ell}}\left(\mu,\tfrac{1}{2} + iu\right)}{\zeta_K^*(1 + 2iu)},
     \]
     where $\tilde{c}_{\kappa_{\ell}}\left(\mu,\tfrac{1}{2} + iu\right)$ is a Dirichlet polynomial in $iu$ and
     \[
       \zeta_K^*(s) = \left(\tfrac{3}{4\pi^2}\right)^{\frac{s}{2}} \Gamma(s) \zeta_K(s)
     \]
     is the completed zeta function for $K$. We let $\overline{\tilde{c}_{\kappa_{\ell}}}\left(\mu,\tfrac{1}{2} + iu\right)$ be a Dirichlet polynomial in $iu$ such that
     \[
       \overline{\tilde{c}_{\kappa_{\ell}}}\left(\mu,\tfrac{1}{2} - iu\right) = \overline{\tilde{c}_{\kappa_{\ell}}\left(\mu,\tfrac{1}{2} + iu\right)},
      \]
      so that
      \[
        \overline{c_{\kappa_{\ell}}\left(\mu,\tfrac{1}{2} + iu\right)} = \frac{\overline{\tilde{c}_{\kappa_{\ell}}}\left(\mu,\tfrac{1}{2} - iu\right)}{\zeta_K^*(1 - 2iu)}.
      \]
      Substituting in the Fourier expansion of the Eisenstein series yields
      \begin{align*}
        &\mathcal E_{\kappa_{\ell}}(z;\mu,s)
        = \frac{\delta_{\kappa_{\ell},\infty}}{2\pi i} \int\limits_{\textup{Re}(w) = 0} \frac{\overline{\tilde{c}_{\infty}}\left(\mu,\tfrac{1}{2} - w\right)}{\zeta_K^*(1 - 2w)}\Gamma(2s - 1 + 2w) \Gamma(2s - 1 - 2w) \\ &\hskip 325pt \cdot y^{1 + 2w}\, dw \\ &\hskip 80pt + \frac{1}{2\pi i} \int\limits_{\textup{Re}(w) = 0} \frac{\overline{\tilde{c}_{\kappa_{\ell}}}\left(\mu,\tfrac{1}{2} - w\right) c_{\kappa_{\ell}}\left(0,\tfrac{1}{2} + w\right)}{\zeta_K^*(1 - 2w)}\Gamma(2s - 1 + 2w) \\ &\hskip 247pt \cdot \Gamma(2s - 1 - 2w) y^{1 - 2w} \,dw \\ &\hskip 80pt + \frac{1}{2\pi} \int\limits_{-\infty}^{\infty} \frac{\overline{\tilde{c}_{\kappa_{\ell}}\left(\mu,\tfrac{1}{2} + iu\right)}}{\zeta_K^*(1 - 2iu)}\Gamma(2s - 1 + 2iu) \Gamma(2s - 1 - 2iu) \\ &\hskip 130pt \cdot \sum_{0 \neq \nu \in \lambda^{-3}\mathcal O_K} c_{\kappa_{\ell}}\left(\nu,\tfrac{1}{2} + iu\right) yK_{2iu}(4\pi|\nu|y) e(\nu x) \,du.
      \end{align*}
      We have the functional equation $\zeta_K^*(1 - 2w) = \zeta_K^*(2w)$, so
      \[
        \frac{\overline{\tilde{c}_{\kappa_{\ell}}}\left(\mu,\tfrac{1}{2} - w\right) c_{\kappa_{\ell}}\left(0,\tfrac{1}{2} + w\right)}{\zeta_K^*(1 - 2w)} = \frac{\overline{\tilde{c}_{\kappa_{\ell}}}\left(\mu,\tfrac{1}{2} - w\right) \tilde{c}_{\kappa_{\ell}}\left(0,\tfrac{1}{2} + w\right)}{\zeta_K^*(1 + 2w)},
      \]
      where $\tilde{c}_{\kappa_{\ell}}(0,s)$ is a Dirichlet polynomial such that
      \[
        c_{\kappa_{\ell}}(0,s) = \frac{\zeta_K^*(2s - 1)}{\zeta_K^*(2s)}\tilde{c}_{\kappa_{\ell}}(0,s).
      \]
      The series in the third integral is absolutely convergent, and its sum is $O\left(e^{-2\pi y}\right)$. Move the line of integration for the first two integrals, passing over the simple poles at $w = -s + \frac{1}{2}$ and $w = s - \frac{1}{2}$, respectively, and no other poles.

      We explicitly analyze the first integral; the argument for the second one is analogous. We have
      \begin{align*}
    &    \frac{1}{2\pi i} \int\limits_{\textup{Re}(w) = 0} \frac{\overline{\tilde{c}_{\infty}}\left(\mu,\tfrac{1}{2} - w\right)}{\zeta_K^*(1 - 2w)}\Gamma(2s - 1 + 2w) \Gamma(2s - 1 - 2w) y^{1 + 2w} \,dw \\
    &
    \hskip 17pt
        = \frac{1}{2\pi i}\hskip-4pt \int\limits_{\textup{Re}(w) = -\sigma + \frac{1}{2} + \varepsilon}
        \hskip-12pt
         \frac{\overline{\tilde{c}_{\infty}}\left(\mu,\tfrac{1}{2} - w\right)}{\zeta_K^*(1 - 2w)}\Gamma(2s - 1 + 2w)  \cdot \Gamma(2s - 1 - 2w) y^{1 + 2w} dw 
         \\ 
         &\hskip 243pt 
         + 2\frac{\overline{\tilde{c}_{\infty}}(\mu,s)}{\zeta_K^*(2s)}\Gamma(4s - 2) y^{2 - 2s}.
      \end{align*}
       We compute
      \begin{align} \label{FundDomainIntegral}
    &    \left\langle 2\frac{\overline{\tilde{c}_{\infty}}(\mu,s)}{\zeta_K^*(2s)}\Gamma(4s - 2) y^{2 - 2s},\;\,\theta\overline{\theta_d} \right\rangle 
        = \int\limits_{\Gamma(9d^2) \backslash \mathfrak h^3} 2\frac{\overline{\tilde{c}_{\infty}}(\mu,s)}{\zeta_K^*(2s)}\Gamma(4s - 2) y^{2 - 2s}
        \nonumber\\
        &
        \hskip 100pt
         \cdot\left(\sigma y^{\frac{2}{3}} + \sum_{\mu \in \lambda^{-3}\mathcal O_K} \overline{\tau(\mu)} yK_{\frac{1}{3}}(4\pi|\mu|y) e^{-4\pi i\textup{Re}(\mu x)}\right)
         \nonumber\\
         &
         \hskip 100pt 
          \cdot \left(\sigma y^{\frac{2}{3}} + \sum_{\nu \in \lambda^{-3}\mathcal O_K} \tau(\nu) yK_{\frac{1}{3}}(4\pi|\nu|y) e^{4\pi i\textup{Re}(\nu dx)}\right) \frac{dxdy}{y^3}
          \nonumber \\
         & \nonumber
         \end{align}
         
         Continuing the computation we obtain
         \begin{align}
     & \left\langle 2\frac{\overline{\tilde{c}_{\infty}}(\mu,s)}{\zeta_K^*(2s)}\Gamma(4s - 2) y^{2 - 2s},\;\,\theta\overline{\theta_d} \right\rangle    =\; 2\frac{\overline{\tilde{c}_{\infty}}(\mu,s)}{\zeta_K^*(2s)}\Gamma(4s - 2) 
       \nonumber \\ \nonumber 
        &
        \hskip 50pt 
        \cdot \int\limits_{\Gamma(9d^2) \backslash \mathfrak h^3} \left[\hskip-36pt\phantom{\int\limits_{\Gamma(9d^2) \backslash \mathfrak h^3}}\sigma^2y^{\frac{1}{3} - 2s} + \sigma y^{\frac{2}{3} - 2s} \sum_{\mu \in \lambda^{-3}\mathcal O_K} \overline{\tau(\mu)} K_{\frac{1}{3}}(4\pi|\mu|y) e^{-4\pi i\textup{Re}(\mu x)}\right. \\
        &
        \hskip 80pt
          + \sigma y^{\frac{2}{3} - 2s}\hskip-8pt \sum_{\nu \in \lambda^{-3}\mathcal O_K} \hskip-5pt \tau(\nu) K_{\frac{1}{3}}(4\pi|\nu|y) e^{4\pi i\textup{Re}(\nu dx)}
           \\ \nonumber &\hskip 108pt 
           + y^{1 - 2s} \left(\,\sum_{\mu \in \lambda^{-3}\mathcal O_K} \hskip-8pt \overline{\tau(\mu)} K_{\frac{1}{3}}(4\pi|\mu|y) e^{-4\pi i\textup{Re}(\mu x)}\right) \\ &\hskip 145pt \cdot  \left.\left(\sum_{\nu \in \lambda^{-3}\mathcal O_K} \tau(\nu) K_{\frac{1}{3}}(4\pi|\nu|y) e^{4\pi i\textup{Re}(\nu dx)}\right)\phantom{\int\limits_{\Gamma(9d^2) \backslash \mathfrak h^3}}\hskip-38pt \right] dxdy.
      \nonumber \end{align}
   
      To bound the above integrals we  let
    $$\Gamma(9d^2)\backslash\mathfrak h^3 =  D(\eta) \cup B(\eta)$$
    where for $\eta >0$ (sufficiently large) we define the Siegel domain
     $$D(\eta) := \Big \{x+jy \in \mathfrak h^3 \;\Big |\; x\in \mathbb C/\big(9d^2\mathcal O_K\big ), \; y>\eta\Big\}$$
     and where 
      $B(\eta)$ is a compact set which can be thought of as the bottom of the fundamental domain. Then it is clear that
   $$\int\limits_{\Gamma(9d^2) \backslash \mathfrak h^3} \; = \; \int\limits_{D(\eta)} \; + \; \int\limits_{B(\eta)}.$$   
     First of all, the integral in (\ref{FundDomainIntegral}) restricted to the compact domain $B(\eta)$ is bounded for $\varepsilon<\text{\rm Re}(s)<1$. Next, the integral in (\ref{FundDomainIntegral}) restricted to the the Siegel domain $D(\eta)$ is given by
 \begin{align*}
 &
 \int\limits_{D(\eta)} \left[\sigma^2y^{\frac{1}{3} - 2s} + \sigma y^{\frac{2}{3} - 2s} \sum_{\mu \in \lambda^{-3}\mathcal O_K} \overline{\tau(\mu)} K_{\frac{1}{3}}(4\pi|\mu|y) e^{-4\pi i\textup{Re}(\mu x)}\right. \\
        &
        \hskip 50pt
         + \sigma y^{\frac{2}{3} - 2s}\hskip-8pt \sum_{\nu \in \lambda^{-3}\mathcal O_K} \hskip-5pt \tau(\nu) K_{\frac{1}{3}}(4\pi|\nu|dy) e^{4\pi i\textup{Re}(\nu dx)} \\
         &
         \hskip 70pt
                  + y^{1 - 2s} \left(\,\sum_{\mu \in \lambda^{-3}\mathcal O_K} \hskip-8pt \overline{\tau(\mu)} K_{\frac{1}{3}}(4\pi|\mu|y) e^{-4\pi i\textup{Re}(\mu x)}\right)
                   \\ & 
                   \hskip 130pt 
       \left.  \hskip-7pt \cdot \left(\sum_{\nu \in \lambda^{-3}\mathcal O_K} \hskip-7pt\tau(\nu) K_{\frac{1}{3}}(4\pi|\nu|dy) e^{4\pi i\textup{Re}(\nu dx)}\right)\phantom{\int\limits_{\Gamma(9d^2) \backslash \mathfrak h^3}}\hskip-38pt \right] dxdy\\
       &
       \\
       &
       \hskip 23pt
         = \textup{Vol}\left(\mathbb C / 9d^2\mathcal O_K\right) \\ &\hskip 40pt \cdot \int\limits_{\eta}^{\infty} \left(\sigma^2y^{\frac{1}{3} - 2s} + y^{1 - 2s} \hskip-8pt \sum_{\nu \in \lambda^{-3}\mathcal O_K} \hskip -8pt \tau(\nu) \overline{\tau(d\nu)} K_{\frac{1}{3}}(4\pi|\nu|y) K_{\frac{1}{3}}(4\pi|\nu|dy)\right) dy. 
 \end{align*}

 If $\textup{Re}(s) > \frac{2}{3}$, then $y^{\frac{1}{3} - 2s}$ is integrable, and its integral is $\frac{1}{\frac{4}{3} - 2s}\eta^{\frac{4}{3} - 2s}$. This expression can be analytically continued to all $s \neq \frac{2}{3}$. The summands in the second term each have a product of two Bessel functions that decay exponentially, so as in the proof of Proposition \ref{ThetaInnerProductBound}  their integrals are asymptotically smaller than the first term and so can be disregarded. Thus for any $\varepsilon > 0$, for $\textup{Re}(s) > \varepsilon$ and $\left|s - \frac{2}{3}\right| > \varepsilon$, we have the bound
      \[
        \left\langle 2\frac{\overline{\tilde{c}_{\infty}}(\mu,s)}{\zeta_K^*(2s)}\Gamma(4s - 2) y^{2 - 2s}, \;\,\theta\overline{\theta_d} \right\rangle \ll_{\varepsilon} \,|s|^{4\text{\rm Re}(s) - \frac{3}{2}} e^{-2\pi|s|}.
      \]
      The same bound holds for the inner product of the residue from the second integral with $\theta\overline{\theta_d}$. By the same argument as in \cite{https://doi.org/10.48550/arxiv.2109.10434}, the absolute values of the shifted integrals are less than a constant times $y^{2 - 2\text{\rm Re}(s)}$, so their inner product with $\theta\overline{\theta_d}$ converges whenever $\text{\rm Re}(s) > 0$ and is less than a constant times $\eta^{\frac{4}{3} - 2\text{\rm Re}(s)}$.

      Therefore, by using the above bounds and Stirling's formula, the absolute value of each individual summand in the definition of $\mathcal E(s)$ is less than a constant times $|s|^{2\text{\rm Re}(s) - \frac{1}{2}} e^{-\pi|s|}$. Because the number of such summands is finite and constant, we thus have
      \[
        \mathcal E(s) \ll_{\varepsilon}\, |s|^{2\text{\rm Re}(s) - \frac{1}{2}} e^{-\pi|s|}.
      \]

        From the above computations, we see that the summand in $\mathcal E(s)$ corresponding to the pole $\kappa_{\ell}$ has a simple pole at $s = \frac{2}{3}$ and no other poles with $\textup{Re}(s) > 0$, and its residue at $s = \frac{2}{3}$ is
        \[
          \textup{Vol}\left(\mathbb C / 9d^2\mathcal O_K\right) \sigma^2\frac{\Gamma\left(\frac{2}{3}\right)}{\zeta_K^*\left(\frac{4}{3}\right)} \bigg(\delta_{\kappa_{\ell},\infty} \left(-\overline{\tilde{c}_{\infty}}\left(1,\tfrac{2}{3}\right)\right) + \overline{\tilde{c}_{\kappa_{\ell}}}\left(1,\tfrac{1}{3}\right) \tilde{c}_{\kappa_{\ell}}\left(0,\tfrac{2}{3}\right)\hskip-4pt\bigg).
        \]
        Thus $\mathcal E(s)$ has a possible simple pole at $s = \frac{2}{3}$ and no other poles with $\textup{Re}(s) > 0$, and its residue at $s = \frac{2}{3}$ is
        \[
          \textup{Vol}\left(\mathbb C / 9d^2\mathcal O_K\right) \sigma^2\frac{\Gamma\left(\frac{2}{3}\right)}{\zeta_K^*\left(\frac{4}{3}\right)} \left(-\overline{\tilde{c}_{\infty}}\left(1,\tfrac{2}{3}\right) + \sum_{\ell = 1}^r \overline{\tilde{c}_{\kappa_{\ell}}}\left(1,\tfrac{1}{3}\right) \tilde{c}_{\kappa_{\ell}}\left(0,\tfrac{2}{3}\right)\right).
        \]
 \end{proof}

\vskip 10pt
\section{Geometric side of the trace formula} \label{GeometricSideSection}
\vskip 5pt

The geometric side of the trace formula is obtained by computing the inner product $\Big\langle P_1(*,s), \,\theta\, \overline{\theta_d}\Big\rangle
$ with the Rankin-Selberg method and then 
using the Fourier expansions of the theta functions. The  geometric side  is given in the following theorem.

\begin{theorem} 
\label{Sd(s)Theorem} 
Fix $\varepsilon>0$ and let $s\in\mathbb C$ with $\text{\rm Re}(s)>1+\varepsilon.$ Then we have
\begin{align*}
    &\Big\langle P_1(*,s), \;\theta\, \overline{\theta_d}\Big\rangle
  = \text{\rm Vol}\Big(\mathbb C / \left(9d^2 \mathcal O_K\right)\Big) \left(\frac{3^{\frac{11}{2}} 2^{-6s}}{ \pi^{2s - \frac{5}{6}}} \frac{\Gamma(2s) \Gamma\left(2s - \frac{2}{3}\right)}{\Gamma\left(2s + \frac{1}{6}\right)} + S_d(s)\right)
  \end{align*}
where
\begin{align*} &
S_d(s)  =  \frac{2^{-6s}\pi^{-2s+\frac12}\;\Gamma(2s)^2}{4\pi i\;\Gamma(s+\frac12)\,\Gamma(s)}  \int\limits_{\text{\rm Re}(w)=-1-\varepsilon} \hskip-10pt
 \frac{\Gamma(s+\frac12+w) \Gamma(s+w) \Gamma(-w)}{\Gamma(2s+\frac12+w)} \\ &\hskip 80pt \cdot \sum_{\nu\in\lambda^{-3}\mathcal O_K} \frac{\tau(\nu) \overline{ \tau(1+d\nu)}}{\big(2|\nu|\cdot|1+d\nu|\big)^{-w} } \int\limits_0^1 \left((1-t)^{-\frac43-w} +(1-t)^{-\frac23-w}\right) \\ &\hskip 195pt \cdot \left((a_d(\nu)+1)\cdot (1-t)+\tfrac{ t^2}{2}  \right)^{w} \,dt\,dw.
\end{align*}
The poles of $S_d(s)$ with $\textup{Re}(s) > 0$ are precisely given by the poles of the expression $\Big\langle P_1(*,s), \;\theta\, \overline{\theta_d}\Big\rangle$ together with a simple pole at $s = \frac{1}{3}$. That is, the poles of $S_d(s)$ with $\textup{Re}(s) > 0$ are precisely a double pole at $s = \frac{1}{2}$; possible simple poles at $s = \frac{1}{2} \pm it_j$ for $t_j \neq 0$, where the poles at $\frac{1}{2} \pm it_j$ occur if and only if $\big\langle u_j,\theta\overline{\theta_d} \big\rangle \neq 0$; a possible simple pole at $s = \frac{2}{3}$; and a simple pole at $s = \frac{1}{3}$. 
\vskip 5pt
The residue of $S_d(s)$ at $s = \frac{1}{3}$ is $-3^{\frac{11}{2}}2^{-3}\pi^{\frac{1}{6}}\frac{\Gamma\left(\frac{2}{3}\right)}{\Gamma\left(\frac{5}{6}\right)}$, while the residue of $S_d(s)$ at each of the poles of $\Big\langle P_1(*,s), \;\theta\, \overline{\theta_d}\Big\rangle$ is $\frac{1}{\textup{Vol}\left(\mathbb C / 9d^2\mathcal O_K\right)}$ times the residue of $\Big\langle P_1(*,s), \;\theta\, \overline{\theta_d}\Big\rangle$ at that pole.
\end{theorem} 
\begin{proof} The proof will be given in four steps.

\vskip 15pt 
 \noindent
 $\underline{\text{{\bf STEP 1}}}${\bf : Rankin-Selberg method.}
\vskip 5pt 
 We first prove
 \begin{equation} \label{SDSFirst Form}
\boxed{S_d(s) = \sum_{\nu\in\lambda^{-3}\mathcal O_K} \tau(\nu) \overline{\tau(1+d\nu)} \int\limits_0^\infty K_\frac13\left(4\pi|\nu| y\right)
K_\frac13\left(4\pi \big|1+d\nu\big| y\right)\, e^{-4\pi y} y^{2s} \;\frac{dy}{y}.}
\end{equation}
 Unraveling the Poincar\'e series $P_1(*,s)$ with the Rankin-Selberg method we see that
  \begin{align*}
   \Big\langle P_1(*,s), \;\theta\, \overline{\theta_d}\Big\rangle & \; = \int\limits_{\Gamma(9d^2) \backslash \mathcal H^3} P_1(z,s) \overline{\theta(z)} \theta(dz)\, \frac{dxdy}{y^3} \;\;
    \\ &= 
    \hskip-8pt\int\limits_{\Gamma_{\infty}(9d^2) \backslash \mathcal H^3} \hskip-12pt y^{2s} e^{-4\pi y} e^{4\pi i\text{\rm Re}(x)} \overline{\theta(z)} \theta(dz)\, \frac{dxdy}{y^3}  
       \\
       &
         = \int\limits_{x \in \mathbb C/ (9d^2 \mathcal O_K)} \int\limits_{y=0}^\infty y^{2s + \frac{4}{3}} e^{-4\pi y} e^{4\pi i\text{\rm Re}(x)} \\
    &\hskip 40pt \cdot \left(\sigma + \sum_{\mu \in \lambda^{-3} \mathcal O_K} \overline{\tau(\mu)} y^{\frac{1}{3}} K_{\frac{1}{3}}(4\pi|\mu|y) e^{-4\pi i\text{\rm Re}(\mu x)}\right) 
    \\ 
    &\hskip 40pt
    \cdot \left(\sigma + \sum_{\nu \in \lambda^{-3} \mathcal O_K} \tau(\nu) y^{\frac{1}{3}} K_{\frac{1}{3}}(4\pi|\nu|y) e^{4\pi i\text{\rm Re}(\nu dx)}\right) \frac{dxdy}{y^3}. 
  \end{align*}

 Continuing the computation,
  \begin{align*}
     &\Big\langle P_1(*,s), \;\theta\, \overline{\theta_d}\Big\rangle
    \\ &\hskip 30pt = \text{Vol}\left(\mathbb C / \left(9d^2 \mathcal O_K\right)\right) \left(\overline{\tau(1)} \,\sigma \int\limits_{y=0}^\infty y^{2s + \frac{5}{3}} e^{-4\pi y} K_{\frac{1}{3}}(4\pi y)\, \frac{dy}{y^3} + A 
 \right.   \\
    &
    \hskip 60pt
    + \sum_{\underset{\mu - d\nu = 1}{\scriptscriptstyle{\mu,\nu \in \lambda^{-3} \mathcal O_K}}} \overline{\tau(\mu)} \tau(\nu) \left.\int\limits_{y=0}^\infty y^{2s + 2} e^{-4\pi y} K_{\frac{1}{3}}(4\pi|\mu|y) K_{\frac{1}{3}}(4\pi|\nu|y)\, \frac{dy}{y^3}\right)
    \\
    &\hskip 30pt
    = \text{Vol}\left(\mathbb C / \left(9d^2 \mathcal O_K\right)\right) \left(27\sigma \int\limits_{y=0}^\infty y^{2s - \frac{4}{3}} e^{-4\pi y} K_{\frac{1}{3}}(4\pi y)\, dy  \;+ \; A \right.\\
    &
    \hskip 60pt
    + \sum_{\underset{\mu - d\nu = 1}{\scriptscriptstyle{\mu,\nu \in \lambda^{-3} \mathcal O_K}}} \overline{\tau(\mu)} \tau(\nu) \left.\int\limits_{y=0}^\infty y^{2s - 1} e^{-4\pi y} K_{\frac{1}{3}}(4\pi|\mu|y) K_{\frac{1}{3}}(4\pi|\nu|y)\, dy\right),
  \end{align*}
where $A$ is defined as follows. If $-\frac{1}{d} \not\in \lambda^{-3} \mathcal O_K$, then we directly have $A = 0$ because there are no summands containing a nonzero integral since no $\nu \in \lambda^{-3} \mathcal O_K$ satisfies $1 + d\nu = 0$. Since $d$ is a (cubefree) positive integer not equal to 1, then $-\frac{1}{d} \in \lambda^{-3} \mathcal O_K$ if and only if $d = 3$. Thus if $d \neq 3$, $A = 0$, and if $d = 3$,
\[
  A = \tau(-1 / 3) \,\sigma \int\limits_{y=0}^\infty y^{2s - \frac{4}{3}} e^{-4\pi y} K_{\frac{1}{3}}\left(\frac{4\pi y}{3}\right) dy.
\]
Because $-\frac{1}{3} = \lambda^{-2}$, which is not in any of the forms specified in the definition of $\tau$, $\tau\left(-\frac{1}{3}\right) = 0$. Thus $A = 0$ in this case as well. Therefore, $A = 0$ in all cases, so we can drop it from the sum. Thus we have

\begin{equation}\label{innerproduct}
  \begin{aligned}
    &\Big\langle P_1(*,s), \;\theta\, \overline{\theta_d}\Big\rangle  = \text{Vol}\left(\mathbb C / 9d^2 \mathcal O_K\right) \left(27\sigma \int\limits_{y=0}^\infty  y^{2s - \frac{4}{3}} e^{-4\pi y} K_{\frac{1}{3}}(4\pi y)\, dy\right. 
    \\ 
    &\hskip 60pt
    + \sum_{\underset{\mu - d\nu = 1}{\scriptscriptstyle{\mu,\nu \in \lambda^{-3} \mathcal O_K}}} \overline{\tau(\mu)} \tau(\nu) \left.\int\limits_{y=0}^\infty  y^{2s - 1} e^{-4\pi y} K_{\frac{1}{3}}(4\pi|\mu|y) K_{\frac{1}{3}}(4\pi|\nu|y)\, dy\right).
  \end{aligned}
\end{equation}

The first integral equals

\begin{equation}\label{firstintegral}
  \begin{aligned}
    \int\limits_{0}^\infty  y^{2s - \frac{4}{3}} e^{-4\pi y} K_{\frac{1}{3}}(4\pi y)\, dy 
    &
      = \frac{\sqrt{\pi}}{ (8\pi)^{2s - \frac{1}{3}} } \frac{\Gamma(2s) \Gamma\left(2s - \frac{2}{3}\right)}{\Gamma\left(2s + \frac{1}{6}\right)} \cdot F\left(2s,\tfrac{5}{6};2s + \tfrac{1}{6};0\right)
    \\
    &
    = 2^{-6s + 1} \pi^{-2s + \frac{5}{6}} \frac{\Gamma(2s) \Gamma\left(2s - \frac{2}{3}\right)}{\Gamma\left(2s + \tfrac{1}{6}\right)},
  \end{aligned}
\end{equation}
where $F$ is the Gaussian hypergeometric function; the final expression is holomorphic on $\left\{s \in \mathbb C : \text{\rm Re}(s) > \frac{1}{3}\right\}$ and has a simple pole at $s = \frac{1}{3}$ with residue $2^{-1} \pi^{\frac{1}{6}} \frac{\Gamma\left(\frac{2}{3}\right)}{\Gamma\left(\frac{5}{6}\right)}$.

\vskip 5pt
The sum on the right side of (\ref{innerproduct}) is $S_d(s)$. This establishes (\ref{SDSFirst Form}).

\vskip 15pt 
 \noindent
 $\underline{\text{{\bf STEP 2}}}${\bf : Recomputing $S_d(s)$ with hypergeometric functions.}
\vskip 10pt
In this step we prove that 

\begin{equation}
\label{Sd(s)SecondForm}
\boxed{\begin{aligned}\displaystyle &S_d(s)  = \frac{2^{-6s}\pi^{-2s+\frac12}\Gamma(2s)^2}{\Gamma(s+\frac12)\Gamma(s)} \sum_{\nu\in\lambda^{-3}\mathcal O_K} \frac{\tau(\nu) \overline{\tau(1+d\nu)}}{2\pi i}\\ &\hskip 35pt \displaystyle \cdot \int\limits_0^\infty \int\limits_{\text{\rm Re}(w)=-1-\varepsilon}
\hskip-10pt\frac{\Gamma(s+\frac12+w) \Gamma(s+w) \Gamma(-w)}{\Gamma(2s+\frac12+w)}\\
&\hskip 35pt \displaystyle
\cdot \Big(|\nu|^2 + |1+d\nu|^2 - 1 \;+\; 2|\nu|\, |1+d\nu|\cdot \cosh(u)   \Big)^w \cosh\left(\frac{u}{3} \right)\;dw \,du.\end{aligned}}
\end{equation}

\begin{proof}

In the next proposition we show that the integral of a product of two K-Bessel functions  given on the right hand side of (\ref{SDSFirst Form}) can be expressed as an integral of a Gaussian hypergeometric function.

\begin{proposition}\label{identity1}
  For $m,n \in \mathbb R_{>0}$, $a \in \mathbb C$, and $s \in \mathbb C$ with $\text{\rm Re}(s) > |\text{\rm Re}(a)|$,
  
  \begin{align*}
    \int\limits_0^{\infty} K_a(my) K_a(ny) e^{-y} y^{2s}\, \frac{dy}{y} 
   &\; = \; \frac{\sqrt{\pi}}{2^{2s}} \frac{\Gamma(2s)^2}{\Gamma\left(2s + \frac{1}{2}\right)} \\ &
   \hskip -10pt
   \cdot \int\limits_0^{\infty} F\left(s + \frac{1}{2},s;2s + \frac{1}{2};1 - \alpha(u)^2\right) \cosh(au)\, du,
  \end{align*}
where $\alpha(u) = \Big(m^2 + n^2 + 2mn\cosh(u)\Big)^{\frac{1}{2}}$.
\end{proposition}

\begin{proof}
  This identity follows from an analogous argument to the one in the first part of the proof of Lemma 1 in \cite{MR1319517} because all of the identities that are used in that argument still hold, as does the way of combining them, with the only potential issue being the convergence of the integral on the right. We now show that the integrals on the left and right sides of the identity in Proposition \ref{identity1} converge absolutely if $\text{\rm Re}(s) > |\text{\rm Re}(a)|$.
  \begin{lemma}\label{identity2}
    For $m,n \in \mathbb R_{>0}$, $a \in \mathbb C$, and $y \in \mathbb R_{>0}$,
    \[
  K_a(my) K_a(ny) = \int\limits_0^{\infty} K_0\left(\Big(m^2 + n^2 + 2mn\cosh(u)\Big)^{\frac{1}{2}} y\right) \cosh(au)\, du.
\]
\end{lemma}
\begin{proof}
  First, use the identity
  \[
  K_r(X) K_r(x) = \frac{1}{2} \int\limits_{-\infty}^{\infty} \int\limits_0^{\infty} e^{-2rT - \frac{Xx}{v} \cosh(2T)} e^{-\left(\frac{v}{2} - \frac{X^2 + x^2}{2v}\right)} \,\frac{dv}{v}\, dT,
\]
which holds for any $r \in \mathbb C$ and $X,x \in \mathbb R_{>0}$ (see page 440 in \cite{MR0010746}), and set $X = my$, $x = ny$, and $r = a$, giving
\[
  K_a(my) K_a(ny) = \frac{1}{2} \int\limits_{-\infty}^{\infty} \int\limits_0^{\infty} e^{-2a T - \frac{mny^2}{v} \cosh(2T)} e^{-\left(\frac{v}{2} - \frac{\left(m^2 + n^2\right) y^2}{2v}\right)} \,\frac{dv}{v}\, dT.
\]
Now use the identity
\[
  K_r(z) = \frac{1}{2} \left(\frac{1}{2} z\right)^r \int\limits_0^{\infty} e^{-t - \frac{z^2}{4t}} \frac{dt}{t^{r + 1}},
\]
which holds for any $r \in \mathbb C$ and  $z \in \mathbb C$ with $\text{\rm Re}\left(z^2\right) > 0$ (see page 183 in \cite{MR0010746}), and set $z = \left(m^2 + n^2 + 2mn\cosh(u)\right)^{\frac{1}{2}} y$ and $r = 0$, giving
\[
  K_0\left(\left(m^2 + n^2 + 2mn\cosh(u)\right)^{\frac{1}{2}} y\right) = \frac{1}{2} \int\limits_0^{\infty} e^{-t - \frac{\left(m^2 + n^2 + 2mn\cosh(u)\right) y^2}{4t}} \,\frac{dt}{t}.
\]
Combining those two identities proves the lemma. 
\end{proof}

It follows from Lemma \ref{identity2} that
\begin{equation}\label{identity3}
  \int\limits_0^{\infty} K_a(my) K_a(ny) e^{-y} y^{2s} \frac{dy}{y} = \int\limits_0^{\infty} \int\limits_0^{\infty} K_0(\alpha(u) y) \cosh(au) e^{-y} y^{2s} \,du \frac{dy}{y},
\end{equation}
provided that the integrals converge.

We now verify that the integral on the left side of (\ref{identity3}) converges for $\text{\rm Re}(s) > |\text{\rm Re}(a)|$ as follows. For any $a \in \mathbb C$, $K_a(y)$ has the following asymptotic bounds (see page 227 in \cite{MR1325466}). 
 
 \begin{align*} & K_a(y) \ll y^{-\frac{1}{2}} e^{-y}\quad \big(\text{for} \; y > 1 + |a|^2\big).
 \\
 &
 \\
 & 
 K_a(y) \ll y^{|\text{\rm Re}(a)|} + y^{-|\text{\rm Re}(a)|}\quad \text{for}\; \big(y < 1 + |a|^2\big).
 \end{align*}

  Without loss of generality, suppose that $m \geq n$. Then
  \begin{align*}
    \int\limits_0^{\infty} K_a(my) K_a(ny) e^{-y} y^{2s} \,\frac{dy}{y}
  &
    \; \ll \; m^{|\text{\rm Re}(a)|} n^{|\text{\rm Re}(a)|} \int\limits_0^{\frac{1 + |a|^2}{m}} e^{-y} y^{2s + 2 |\text{\rm Re}(a)|} \,\frac{dy}{y} 
    \\ 
    &
    \hskip -74pt
    + \Big(m^{-|\text{\rm Re}(a)|} n^{|\text{\rm Re}(a)|} \;\, + \;\, m^{|\text{\rm Re}(a)|} n^{-|\text{\rm Re}(a)|}\Big) 
    \int\limits_0^{\frac{1 + |a|^2}{m}} e^{-y} y^{2s} \,\frac{dy}{y} 
    \\ 
    &
     \hskip -74pt
    + m^{-|\text{\rm Re}(a)|} n^{-|\text{\rm Re}(a)|} \int\limits_0^{\frac{1 + |a|^2}{m}} e^{-y} y^{2s - 2 |\text{\rm Re}(a)|} \frac{dy}{y} \;\,
    \\ &\hskip -70pt + \;\,  m^{-\frac{1}{2}} n^{|\text{\rm Re}(a)|} \int\limits_{\frac{1 + |a|^2}{m}}^{\frac{1 + |a|^2}{n}} e^{-(1 + m) y} y^{2s - \frac{1}{2} + |\text{\rm Re}(a)|} \,\frac{dy}{y}
     \\
     &
      \hskip -74pt
      + m^{-\frac{1}{2}} n^{-|\text{\rm Re}(a)|} \int\limits_{\frac{1 + |a|^2}{m}}^{\frac{1 + |a|^2}{n}} e^{-(1 + m) y} y^{2s - \frac{1}{2} - |\text{\rm Re}(a)|} \,\frac{dy}{y} 
    \;\, \\ &\hskip -74pt + \;\, m^{-\frac{1}{2}} n^{-\frac{1}{2}} \int\limits_{\frac{1 + |a|^2}{n}}^{\infty} e^{-(1 + m + n) y} y^{2s - 1} \,\frac{dy}{y},
  \end{align*}
which all converge if $\text{\rm Re}(s) \geq |\text{\rm Re}(a)|$.

Next, we make use of the Mellin transform
  \begin{align*}
    \int\limits_0^{\infty} e^{-ay} K_{\nu}(by) y^{s - 1} dy 
    &
    = \frac{\sqrt{\pi}}{(2a)^{s}} \left(\frac{b}{a}\right)^{\nu} \frac{\Gamma(s + \nu) \Gamma(s - \nu)}{\Gamma\left(s + \frac{1}{2}\right)} \\ 
    &
    \hskip 40pt
    \cdot F\left(\frac{s + \nu + 1}{2},\frac{s + \nu}{2};s + \frac{1}{2};1 - \left(\frac{b}{a}\right)^2\right),
  \end{align*}
which holds for any $\nu,a,b,s \in \mathbb C$ such that $\text{\rm Re}(s) > |\text{\rm Re}(\nu)|$ and $\text{\rm Re}(a + b) > 0$ (see page 331 in \cite{MR0061695}) and set $a = 1$, $b = \alpha(u)$, and $\nu = 0$ and substitute $2s$ for $s$. It follows that
\[
  \int\limits_0^{\infty} K_0(\alpha(u) y) e^{-y} y^{2s} \,\frac{dy}{y} =  \frac{\sqrt{\pi}}{ 2^{2s}} \frac{\Gamma(2s)^2}{\Gamma\left(2s + \frac{1}{2}\right)} \,F\left(s + \frac{1}{2},s,2s + \frac{1}{2};1 - \alpha(u)^2\right).
\]
Combining this with (\ref{identity3}) gives the formula in the statement of Proposition \ref{identity1}. In particular, if $a = \frac{1}{3}$, the identity holds for all $s \in \mathbb C$ with $\text{\rm Re}(s) > \frac{1}{3}$.
\end{proof}

We now use the following integral representation of the hypergeometric function (see \cite{MR1575118}).
\begin{proposition}\label{identity4} Fix $r>0.$ Then
  \[
  F(\alpha,\beta;\gamma;z) = \frac{\Gamma(\gamma)}{\Gamma(\alpha) \Gamma(\beta)} \cdot \frac{1}{2\pi i} \int\limits_{\text{\rm Re}(w) = -r} \frac{\Gamma(\alpha + w) \Gamma(\beta + w) \Gamma(-w)}{\Gamma(\gamma + w)} (-z)^w \, dw,
\]
where $|\arg(-z)| < \pi$  and $\text{\rm Re}(\alpha),\text{\rm Re}(\beta) > r$.
\end{proposition}

The formula for $S_d(s)$ given in (\ref{Sd(s)SecondForm})  immediately follows from Propositions \ref{identity1} and \ref{identity4} for $\text{\rm Re}(s) > 1+\varepsilon$ with the choice $r=1+\varepsilon.$ 

\end{proof}

\vskip 10pt 
 \noindent
 $\underline{\text{{\bf STEP 3}}}${\bf : We rewrite $S_d(s)$ using  Picard's integral representation of the Appell hypergeometric function.}
 \vskip 10pt
  
  In this step we complete the proof of the first part of Theorem \ref{Sd(s)Theorem}. 
It is helpful to introduce some additional notation. Let
 \begin{align*}
 Z_{d,\nu}(u) & = |\nu|^2+ |1+d\nu|^2-1\;+\;2|\nu|\cdot|1+d\nu| \cdot\cosh(u)\\
 & =  2|\nu|\cdot|1+d\nu| \cdot \Big(a_d(\nu) \;+\; \cosh(u)\Big),
 \end{align*}
 where $$a_d(\nu) := \;\frac{ |\nu|^2+ |1+d\nu|^2-1}{2|\nu|\cdot|1+d\nu|} = \frac{d^2+1}{2d}\bigg(1+\mathcal O_d\left(|\nu(1+dv)|^{-\frac12}\right)\bigg).$$

\begin{lemma} \label{CoshLemma} Let $a>1$ and $w\in\mathbb C$ with $\text{\rm Re}(w)<-\frac13.$ Then
\begin{align*}
& \int\limits_0^\infty \big(a + \cosh(u)\big)^w \cdot\cosh\left(\frac{u}{3}\right)\,du
 \\ &\hskip 80pt = \frac{(a+1)^{w}}{18w^2-2}\Bigg[(3-9w) \Phi_1(w,a) -(3+9w)\Phi_2(w,a)
 \Bigg],
 \end{align*}
 where
 $$\Phi_1(w,a) := F_1\left(1,\,-w,\,-w,\,-w+\frac23;\; \frac12-\frac12\sqrt{\frac{a-1}{a+1}}, \,  \frac12+\frac12\sqrt{\frac{a-1}{a+1}}  \right),$$
$$\Phi_2(w,a) := F_1\left(1,\,-w,\,-w,\,-w+\frac43;\; \frac12-\frac12\sqrt{\frac{a-1}{a+1}}, \,  \frac12+\frac12\sqrt{\frac{a-1}{a+1}}  \right).$$
\end{lemma} 
\begin{proof} By a Mathematica$^{TM}$ computation we have
\begin{align} \label{CoshIntegral}
& \int\limits_0^\infty \big(a + \cosh(u)\big)^w \cdot\cosh\left(\frac{u}{3}\right)\,du
 = \frac{2^{-w}}{18w^2-2}\\ &\cdot \Bigg[(3-9w) F_1\left(-\tfrac13-w, -w,-w,\tfrac23-w;\; -a+\sqrt{a^2-1}, \; \frac{1}{-a+\sqrt{a^2-1}}  \right)\nonumber\\
 &
 -(3+9w) F_1\left(\tfrac13-w,-w,-w,\tfrac43-w;\; -a+\sqrt{a^2-1}, \; \frac{1}{-a+\sqrt{a^2-1}}  \right)
 \Bigg].
 \nonumber
 \end{align}
 For $|z_1|,|z_2|<1$, we have the identity
 \begin{equation*}
   \boxed{\begin{aligned}\displaystyle &F_1(a,b_1,b_2,c;\;z_1,z_2) = (1-z_1)^{-b_1} (1-z_2)^{-b_2} \\ \displaystyle &\hskip 150pt \cdot F_1\left(c-a,b_1,b_2,c; \frac{z_1}{z_1-1},\;\frac{z_2}{z_2-1} \right).\end{aligned}}
 \end{equation*}
 It follows that
\begin{align*}
&
F_1\left(-\tfrac13-w,\,-w,\,-w,\, \tfrac23-w,\,-a+\sqrt{a^2-1},\, \frac{1}{-a+\sqrt{a^2-1}}    \right)\\
&
\hskip 30pt
= \big(2+2a\big)^{w}F_1\left(1,-w,-w,\tfrac23-w;\; \frac12 -\frac12\sqrt{\frac{a-1}{a+1}},\;  \frac12 +\frac12\sqrt{\frac{a-1}{a+1}} \right),\\
&
F_1\left(\tfrac13-w,\,-w,\,-w,\, \tfrac43-w,\,-a+\sqrt{a^2-1},\, \frac{1}{-a+\sqrt{a^2-1}}    \right)\\
&
\hskip 30pt
= \big(2+2a\big)^{w}F_1\left(1,-w,-w,\tfrac43-w;\; \frac12 -\frac12\sqrt{\frac{a-1}{a+1}},\;  \frac12 +\frac12\sqrt{\frac{a-1}{a+1}} \right).
\end{align*}
Inserting the above identities in (\ref{CoshIntegral}) completes the proof.
\end{proof}

\'Emile Picard \cite{MR1508705} proved the following integral representation of Appell's hypergeometric function (which is valid for $\text{Re}(c) > \text{Re}(\alpha) > 0$):
\begin{align*}&F_1(\alpha,\beta_1,\beta_2,c; z_1,z_2) \\ &\hskip 50pt = \frac{\Gamma(c)}{\Gamma(\alpha)\Gamma(c-\alpha)}\int\limits_0^1 t^{\alpha-1} (1-t)^{c-\alpha-1} (1-z_1t)^{-\beta_1}(1-z_2t)^{-\beta_2}\; dt.\end{align*}

Let 
$$z_1(a) := \frac12-\frac12\sqrt{\frac{a-1}{a+1}}, \qquad z_2(a) :=
\frac12+\frac12\sqrt{\frac{a-1}{a+1}}.$$ 

Note that
\begin{align*}
\big(1-z_1(a)\cdot t\big)\cdot\big(1-z_2(a)\cdot t\big) & = 1-t+\frac14\left(1 -\frac{a-1}{a+1}   \right)t^2\\
& = 1-t+ \frac{t^2}{2(a+1)}.
\end{align*}

 It follows that
\begin{align}\label{Phi-1-2}
\Phi_1(w,a) &:= (-w-\tfrac13)\int_0^1
(1-t)^{-w-\frac43} \left(1-t+\tfrac{t^2}{2(a+1)}   \right)^{w} dt,
\\
 \Phi_2(w,a) &:= (-w+\tfrac13)\int_0^1
(1-t)^{-w-\frac23} \left(1-t+\tfrac{ t^2}{2(a+1)}  \right)^{w} dt.
\nonumber\end{align}
The proof of the first part of Theorem \ref{Sd(s)Theorem} now follows from the definition of $S_d(s)$, Lemma \ref{CoshLemma}, and Picard's formulae 
(\ref{Phi-1-2}).

\vskip 15pt 
 \noindent
$\underline{\text{{\bf STEP 4}}}${\bf : The poles of $S_d(s)$.}
\vskip 5pt

The only pole of $\frac{\Gamma(2s) \Gamma\left(2s - \frac{2}{3}\right)}{\Gamma\left(2s + \frac{1}{6}\right)}$ with $\textup{Re}(s) > 0$ is a simple pole at $s = \frac{1}{3}$.\linebreak Because $\big\langle P_1(*,s), \;\theta\, \overline{\theta_d}\big\rangle$ does not have a pole at $s = \frac{1}{3}$, there is\linebreak no cancellation, so the poles of $S_d(s)$ with $\textup{Re}(s) > 0$ are precisely a simple pole at $s = \frac{1}{3}$ and the poles of $\big\langle P_1(*,s), \;\theta\, \overline{\theta_d}\big\rangle$ with $\textup{Re}(s) > 0$ with the same order as for that function. The computation of the residue of $S_d(s)$ at each of those poles is then an immediate consequence of the formula.
 \end{proof}

\vskip 15pt
\section{Relating $S_d(s)$ to $L_d^\#(s) $}
\label{RelatingSdSSection}

The function
\begin{align*} &
S_d(s)  =  \frac{2^{-6s-1}\pi^{-2s+\frac12}\;\Gamma(2s)^2}{2\pi i\;\Gamma(s+\frac12)\,\Gamma(s)}  \hskip-5pt\int\limits_{\text{\rm Re}(w)=-1-\varepsilon} \hskip-12pt
 \frac{\Gamma(s+\frac12+w) \Gamma(s+w) \Gamma(-w)}{\Gamma(2s+\frac12+w)} 
 \\ &
 \hskip 100pt 
 \cdot \sum_{\nu\in\lambda^{-3}\mathcal O_K} \frac{\tau(\nu) \overline{ \tau(1+d\nu)}}{\big(2|\nu|\cdot|1+d\nu|\big)^{-w} } 
 \\ &
 \hskip 55pt 
 \cdot \int\limits_0^1 \left(t^{-\frac43-w} +t^{-\frac23-w}\right) \left(\big(a_d(\nu)+1\big)\cdot t+\tfrac{ (1-t)^2}{2}  \right)^{w} dt\,dw
   \end{align*}
that occurs on the geometric side of the trace formula (see Theorem \ref{Sd(s)Theorem}) contains the coefficients $\tau(\nu)\overline{\tau(1+d\nu)}$ which occur in the Dirichlet series expansions of the cubic Pell equation L-functions $L_d(s)$ and $L_d^\#(s)$. The integral in $w$ above can be evaluated by shifting the line of integration to the left and picking up residues at $w= -s-\frac12, -s$ which leads to a proof of  the following theorem.

  \begin{theorem} {\bf (Relating $S_d(s)$ to $L_d^\#(s)$)}
  \label{RelatingSd(s)} Fix $\varepsilon>0$ sufficiently small.   Then
  \begin{align*}
  &S_d(s) \,=\,  \frac{2^{-7s-1}}{\pi^{2s-1}}\frac{\Gamma(2s)^2}{\Gamma(s+\frac12)^2}\cdot L_d^\#(s)\; 
  -\;
  \frac{2^{-7s+\frac12}}{\pi^{2s-1}}\frac{\Gamma(2s)^2} {\Gamma(s)^2}\cdot  L_d^{\#}(s+\tfrac12) \; \\ &\hskip 250pt + \;\mathcal O\bigg(|s|^{2\text{\rm Re}(s)+3} e^{-\pi|s|}   \bigg).
  \end{align*}
  The function $L_d^\#(s)$ has meromorphic continuation to $s\in\mathbf C$ with $\text{\rm Re}(s) > \frac13$ and satisfies
    $$|L_d^\#(s)| \ll_\varepsilon \, |s|^3$$
    for $\frac13+\varepsilon < \text{\rm Re}(s)<1$ provided $|s-\rho|>\varepsilon$  for any pole $\rho\in\mathbb C$ of $S_d(s)$. 
    
    \vskip 5pt
    The poles of $L_d^{\#}(s)$ with $\textup{Re}(s) >\frac13$ are precisely the poles of $S_d(s)$ (see Theorem \ref{Sd(s)Theorem}) with each having the same order as the corresponding pole of $S_d(s)$.
    \end{theorem}

 \vskip 5pt
\begin{remark} 
     At each of the poles $\rho$ of $S_d(s)$, the residue of $L_d^{\#}(s)$ at that pole is 
     $$\underset{s=\rho}{\text{\rm Res}}\;L_d^{\#}(s) = \frac{\pi^{2s - 1}}{2^{-7s - 1}}\frac{\Gamma\left(s + \frac{1}{2}\right)^2}{\Gamma(2s)^2} \cdot \underset{s=\rho}{\text{\rm Res}}\; S_d(s).$$ 
\end{remark}

\begin{proof}

 The proof of Theorem \ref{RelatingSd(s)} will be given in several  steps.

\vskip 15pt 
 \noindent
 $\underline{\text{{\bf STEP 1}}}${\bf : We shift the line of integration in the $w$-integral of $S_d(s)$ obtaining two residue terms and a shifted integral term.}
 \vskip 5pt
 \begin{proposition} \label{SdSIdentity3}
Fix $\varepsilon>0$ sufficiently small. Let $s\in\mathbb C$ with $1+\varepsilon<\text{\rm Re}(s) < 1+2\varepsilon.$ Then we have
$$S_d(s) = \mathcal R_1(s) + \mathcal R_2(s) + I(s)$$
 where $\mathcal R_1(s)$, $\mathcal R_1(s),$ and $\mathcal I(s)$ are holomorphic  for $1+\varepsilon<\text{\rm Re}(s) < 1+2\varepsilon$ and are given by
  \begin{align*} 
 \mathcal R_1(s) & := 
   \frac{2^{-6s-1}}{\pi^{2s-1}}\frac{\Gamma(2s)^2}{\Gamma(s+\frac12)^2} \sum_{\nu\in\lambda^{-3}\mathcal O_K} \frac{\tau(\nu) \overline{\tau(1+d\nu)}}{\big(2|\nu|\cdot|1+d\nu| \big)^s}
\\
&
\hskip 68pt
\cdot \int_0^1\left( t^{-\frac43+s} + t^{-\frac23+s}  \right) \left(\big(a_d(\nu)+1\big)\cdot t   +\tfrac{(1-t)^2}{2}\right)^{-s} dt,
\end{align*}
\begin{align*}
 \mathcal R_2(s) &:= -\frac{2^{-6s}}{\pi^{2s-1}}\frac{\Gamma(2s)^2}{\Gamma(s)^2} \sum_{\nu\in\lambda^{-3}\mathcal O_K} 
 \frac{\tau(\nu) \overline{\tau(1+d\nu)}}{\big(2|\nu|\cdot|1+d\nu|\big)^{s+\frac12}}
 \\
&
\hskip 60pt
\cdot \int_0^1 \left( t^{-\frac56+s} + t^{-\frac16+s}  \right) \left((a_d(\nu)+1)\cdot t   +\tfrac{(1-t)^2}{2}\right)^{-s-\frac12} dt,
\end{align*}
\begin{align*}
\mathcal I(s) 
&:= \frac{2^{-6s}\pi^{-2s+\frac12}\;\Gamma(2s)^2}{4\pi i\;\Gamma(s+\frac12)\,\Gamma(s)}  \int\limits_{\text{\rm Re}(w)=-\frac32-3\varepsilon} \hskip-13pt
 \frac{\Gamma(s+\frac12+w) \Gamma(s+w) \Gamma(-w)}{\Gamma(2s+\frac12+w)} \\ &\hskip 53pt \cdot \sum_{\nu\in\lambda^{-3}\mathcal O_K} \frac{\tau(\nu) \overline{ \tau(1+d\nu)}}{\big(2|\nu|\cdot|1+d\nu|\big)^{-w} }\\
 &
 \hskip13pt
 \cdot \int_0^1 \left(t^{-\frac43-w} +t^{-\frac23-w}\right) \left(\big(a_d(\nu)+1\big)\cdot t+\tfrac{ (1-t)^2}{2}  \right)^{w} dt\,dw.
 \end{align*}

\end{proposition}

\begin{proof} Shift the line of integration in the $w$-integral of $S_d(s)$ given in Theorem \ref{Sd(s)Theorem} to the line $\text{\rm Re}(w) = -\frac32-3\varepsilon.$ Since we are assuming that $1+\varepsilon<\text{\rm Re}(s) < 1+2\varepsilon$ we will only pass the two poles of $$\frac{\Gamma(s+\frac12+w) \Gamma(s+w) \Gamma(-w)}{\Gamma\left(2s+\frac12+w\right)}$$ at $w=-s$ and $w=-s-\frac12.$ The residues at these poles are given by $\mathcal R_1(s), \mathcal R_2(s)$, and by Cauchy's residue theorem, the original integral is the sum of the two residues plus the shifted integral $\mathcal I(s).$ 

Note that the sum $$\sum_{\nu\in\lambda^{-3}\mathcal O_K}\frac{ \tau(\nu) \overline{\tau(1+d\nu)}}{ \big(2|\nu|\cdot|1+d\nu|\big)^{-w} }$$ converges absolutely for $\text{\rm Re}(w) <-1$ by Proposition \ref{Ld(s)convergence}. Also Since $a_d(\nu)+1\sim \frac{(d+1)^2}{2d}>1$ it follows 
 that
 the $t$-integral in $\mathcal I(s)$ above is bounded 
 for any $w\in\mathbb C$ with $\text{\rm Re}(w) =-\frac32-3\varepsilon.$ One easily concludes that $\mathcal R_1(s), \mathcal R_2(s),$ and $\mathcal I(s)$ are holomorphic functions for $s\in\mathbb C$ with $1+\varepsilon<\text{\rm Re}(s) < 1+2\varepsilon$ for .
\end{proof}
\vskip 10pt 
 \noindent
 $\underline{\text{{\bf STEP 2}}}${\bf : Holomorphic continuation and bounds for $\mathcal R_2(s)$ and $\mathcal I(s)$ in the region $\boldmath  \frac12+4\varepsilon<\text{\rm \bf Re}(s) < 1+2\varepsilon.$ }
 \vskip 5pt

 \begin{proposition}\label{I(s)Continuation1} The  residue  $\mathcal R_2(s)$ and the integral $\mathcal I(s)$ are holomorphic functions  and satisfy the bounds
$$\left|\mathcal R_2(s)\right| \ll |s|^{2\text{\rm Re}(s)}  e^{-\pi|s|}, \qquad\quad |\mathcal I(s)| \ll_\varepsilon |s|^{\frac32+7\varepsilon}\, e^{-\pi|s|},$$
for $s\in\mathbb C$ with   $\frac12+4\varepsilon <\text{\rm Re}(s) < 1+2\varepsilon$
 \end{proposition}
 
 \begin{proof} 
 By Stirling's asymptotic formula
$$
  \left|\frac{\Gamma(2s)^2}{\Gamma\left(s\right)^2}\right| \ll |s|^{2\text{\rm Re}(s)}e^{-\pi|s|}
$$
for  $\frac12+4\varepsilon <\text{\rm Re}(s) < 1+2\varepsilon$.
Since $a_d(\nu)+1\sim \frac{(d+1)^2}{2d}>1$ it follows 
 that
 the $t$-integral in $\mathcal R_2(s)$ is bounded by an absolute constant
 for any $s\in\mathbb C$ with $\frac12+4\varepsilon <\text{\rm Re}(s) < 1+2\varepsilon$ and any $w\in\mathbb C$ with $\text{\rm Re}(w) =-\frac32-3\varepsilon.$
It immediately follows that $\left|\mathcal R_2(s)\right| \ll |s|^{2\text{\rm Re}(s)}  e^{-\pi|s|}$ in this region.

\vskip 10pt

We have shown in Proposition \ref{SdSIdentity3}  that $\mathcal I(s)$ is  holomorphic in the region $1+\varepsilon<\text{\rm Re}(s) < 1+2\varepsilon$.
 Note that for $\text{\rm Re}(w) = -\frac32-3\varepsilon$ the ratio of Gamma functions
 $$\frac{\Gamma(s+\frac12+w) \Gamma(s+w) \Gamma(-w)}{\Gamma\left(2s+\frac12+w\right)}$$
 is a holomorphic  in the region  $\frac12+4\varepsilon <\text{\rm Re}(s) < 1+2\varepsilon$. We now show that the integral $\mathcal I(s)$ converges absolutely in this region and, as a consequence,  obtain a sharp bound for its growth in this region.

Further, for $w=-\frac32-3\varepsilon +iv$ and $s=\sigma+it$ with $v,t\in\mathbb R$,  Stirling's asymptotic formula implies that   for $\frac12+4\varepsilon <\sigma <1+\varepsilon$ we have
\begin{align*}
  &\left|\frac{\Gamma(s+\frac12+w) \Gamma(s+w) \Gamma(-w)}{\Gamma(2s+\frac12+w)}\right| \\ &\hskip 80pt \;\ll \;\frac{(1 + |t + v|)^{2\sigma - \frac{7}{2} - 6\varepsilon} (1 + |v|)^{1 + 3\varepsilon}}{ (1 + |2t + v|)^{2\sigma - \frac{3}{2} - 3\varepsilon}}\cdot  e^{-\frac{\pi}{2} \big(2|t + v| + |v| - |2t + v|\big)}.
\end{align*}
Let
\begin{equation}\label{I(s)}
\mathcal I(s) = \frac{\pi^{-2s+1}}{2^{6s+1}}\frac{\Gamma(2s)^2}{\Gamma(s+\frac12)\Gamma(s)}\cdot \mathcal I_1(s).
\end{equation}
It follows that 
\begin{align*}
\big|\mathcal I_1(\sigma+it)\big|  &  \ll  \int\limits_{-\infty}^\infty 
 \frac{(1 + |t + v|)^{2\sigma - \frac{7}{2} - 6\varepsilon} (1 + |v|)^{1 + 3\varepsilon}}{ (1 + |2t + v|)^{2\sigma - \frac{3}{2} - 3\varepsilon}}\cdot  e^{-\frac{\pi}{2} \big(2|t + v| + |v| - |2t + v|\big)}\; dv\\
 & = \int\limits_{-\infty}^\infty \;\underset{\text{change of variables $v\mapsto u-t$}}{\underbrace{
 \frac{(1 + |u|)^{2\sigma - \frac{7}{2} - 6\varepsilon} (1 + |u-t|)^{1 + 3\varepsilon}}{ (1 + |u+t|)^{2\sigma - \frac{3}{2} - 3\varepsilon}}\cdot  e^{-\frac{\pi}{2} \big(2|u| + |u-t| - |u+t|\big)}}}\; du
 \end{align*}
 We now assume that $t>0$, the case $t<0$ can be shown by an analogous argument. To bound $\mathcal I_1(s)$ we split the interval of integration into 4 regions:
 $$\int_{-\infty}^\infty = \int_0^t+\int_t^\infty+\int_{-t}^0 + \int_{-\infty}^{-t}.$$
 \vskip 5pt
 {\bf Case 1:} $t\ge 0$ and $0\le u<t$. 
  $$\int\limits_0^t \frac{(1 + u)^{2\sigma - \frac{7}{2} - 6\varepsilon} (1 + t-u)^{1 + 3\varepsilon}}{ (1 + t+u)^{2\sigma - \frac{3}{2} - 3\varepsilon}}\;du \ll (1+t)^{-2-3\varepsilon}.
 $$
 \vskip 5pt
 {\bf Case 2:} $t\ge 0$ and $t\le u.$
\begin{align*}
&\int\limits_t^\infty \frac{(1 + u)^{2\sigma - \frac{7}{2} - 6\varepsilon} (1 +u-t)^{1 + 3\varepsilon}}{ (1 + t+u)^{2\sigma - \frac{3}{2} - 3\varepsilon}}\, e^{-\pi(u-t)}\,du \\ &\hskip 40pt = \int\limits_t^{t+t^\varepsilon}  \frac{(1 + u)^{2\sigma - \frac{7}{2} - 6\varepsilon} (1 +u-t)^{1 + 3\varepsilon}}{ (1 + t+u)^{2\sigma - \frac{3}{2} - 3\varepsilon}}\, e^{-\pi(u-t)}\,du \;+\;\mathcal O\Big(e^{-t^\varepsilon}  \Big)\\
&\hskip 40pt \ll (1+t)^{-2 -3\varepsilon}.
\end{align*}
\vskip 5pt
 {\bf Case 3:} $t \ge 0$ and $-t\le u<0$.
\begin{align*}
&\int\limits_{-t}^0 \frac{(1 - u)^{2\sigma - \frac{7}{2} - 6\varepsilon} (1 +t-u)^{1 + 3\varepsilon}}{ (1 + t+u)^{2\sigma - \frac{3}{2} - 3\varepsilon}}\, e^{2\pi u}\;du \\ &\hskip 40pt = \int\limits_{-t^\varepsilon}^0 \frac{(1 - u)^{2\sigma - \frac{7}{2} - 6\varepsilon} (1 +t-u)^{1 + 3\varepsilon}}{ (1 + t+u)^{2\sigma - \frac{3}{2} - 3\varepsilon}}\, e^{2\pi u}\,du \;+\; \mathcal O\Big(e^{-t^\varepsilon}  \Big)\\
&\hskip 40pt \ll (1+t)^{\frac52-2\sigma +7\varepsilon}.
 \end{align*}
\vskip 5pt
{\bf Case 4:} $t \ge 0$ and $u\le -t$.
\begin{align*}
&\int\limits_{-\infty}^{-t} \frac{(1 - u)^{2\sigma - \frac{7}{2} - 6\varepsilon} (1 +t-u)^{1 + 3\varepsilon}}{ (1 +t-u)^{2\sigma - \frac{3}{2} - 3\varepsilon}}\, e^{\pi (u-t)}\,du \\ &\hskip 40pt = \int\limits_{-t-t^\varepsilon}^{-t} \frac{(1 - u)^{2\sigma - \frac{7}{2} - 6\varepsilon} (1 +t-u)^{1 + 3\varepsilon}}{ (1 +t-u)^{2\sigma - \frac{3}{2} - 3\varepsilon}}\, e^{\pi (u-t)}\,du  \;+\; \mathcal O\Big(e^{-t^\varepsilon}  \Big)\\
&\hskip 40pt \ll (1+t)^{-2 -3\varepsilon}.
 \end{align*}

It immediately follows that for $\frac12+4\varepsilon < \text{\rm Re}(s) <1+2\varepsilon$ that
$$\big|\mathcal I_1(s)\big|\ll |s|^{\frac52-2\text{\rm Re}(s)+7\varepsilon}.$$
Again, 
by Stirling's formula (for $s = \sigma+it$ and $\frac12+4\varepsilon < \sigma <1+\varepsilon$), we have
\[
  \left|\frac{\Gamma(2s)^2}{\Gamma\left(s + \frac{1}{2}\right)^2}\right| \ll |s|^{2\sigma - 1}e^{-\pi|t|}.
\]
The proof of the bound for $\mathcal I(s)$ given in Proposition \ref{I(s)Continuation1} immediately follows from the above two bounds and (\ref{I(s)}).
\end{proof}

\pagebreak
\vskip 15pt 
 \noindent
 $\underline{\text{{\bf STEP 3}}}$ {\bf : Meromorphic  continuation of $\boldmath \mathcal R_1(s) + 2\mathcal R_2(s)$ to $\text{\rm \bf Re}(s)>\frac12.$}
 \vskip 5pt
\begin{proposition} \label{R1plu2R2}
Fix $\varepsilon > 0.$ 
 Then $$S_d(s) = \mathcal R_1(s) + 2\mathcal R_2(s) +\mathcal O\bigg(\Big(|s|^{2\text{\rm Re}(s)} + |s|^{\frac32+7\varepsilon} \Big) e^{-\pi|s|}  \bigg)$$  for  $\frac12+4\varepsilon<\text{\rm  Re}(s) < 1+2\varepsilon$.
 \end{proposition}
\begin{proof}
 The proof of Proposition \ref{R1plu2R2} immediately follows from  Proposition \ref{SdSIdentity3} (which says that
$S_d(s) = \mathcal R_1(s) +\mathcal R_2(s) + \mathcal I(s))$ and Proposition \ref{I(s)Continuation1} (which gives bounds for $\mathcal R_2(s)$ and $\mathcal I(s)$). We already know that $S_d(s)$ has meromorphic continuation to $\text{\rm Re}(s) > 0$ by Theorems \ref{SpectralSideBound} and \ref{Sd(s)Theorem}.\end{proof}

\vskip 15pt 
 \noindent
 $\underline{\text{{\bf STEP 4}}}${\bf : Meromorphic continuation of $\mathcal R_1(s) + 2\mathcal R_2(s)$ to Re$(s)> \varepsilon$.}
 \vskip 5pt
 \begin{proposition}
 \label{R1and2R2prop}
  Fix $\varepsilon >0$ sufficiently small.  Then
 $$S_d(s) = \mathcal R_1(s) + 2\mathcal R_2(s) + \mathcal O\left(  |s|^{\frac52+7\varepsilon}  e^{-\pi|s|}  \right)$$
 for $4\varepsilon <\text{\rm Re}(s) <\frac12+2\varepsilon.$
  \end{proposition}

  \begin{proof}
 To meromorphically continue  $\mathcal R_1(s) + 2\mathcal R_2(s)$ to Re$(s)> \varepsilon$ we first assume that  $s\in\mathbb C$ satisfies $\frac12+\varepsilon <\text{\rm Re}(s)< \frac12+2\varepsilon$. Next, we shift the line of integration in the $w$-integral of $\mathcal I(s)$ from the line $\text{\rm Re}(w) = -\frac32-3\varepsilon$ to the line $\text{\rm Re}(w) = -2-3\varepsilon.$ In doing so we cross possible poles of $\Gamma(s+\frac12+w) \Gamma(s+w)$  at
 $$w =  -2-\alpha\varepsilon,  \;\;-\tfrac32-\alpha\varepsilon, \qquad (1\le\alpha\le 2).$$
 Thus, the only pole that is crossed is the pole at $w=-\frac12-s$ which has $\mathcal R_2(s)$ as the residue.

\vskip 5pt 
  It follows that for $\frac12+\varepsilon <\text{\rm Re}(s)< \frac12+2\varepsilon$ we have $S_d(s) = \mathcal R_1(s) + \mathcal R_2(s) + \mathcal I(s)$ where
\begin{align*} 
&
\mathcal I(s) = \mathcal R_2(s) \; + \;\mathcal I_2(s) \end{align*}
and
 \begin{align*}
& \mathcal I_2(s) := \frac{2^{-6s}\pi^{-2s+\frac12}\;\Gamma(2s)^2}{4\pi i\;\Gamma(s+\frac12)\,\Gamma(s)}  
\int\limits_{\text{\rm Re}(w)=-2-3\varepsilon} \hskip-10pt
 \frac{\Gamma(s+\frac12+w) \Gamma(s+w) \Gamma(-w)}{\Gamma(2s+\frac12+w)} \\ &\hskip 93pt \cdot \sum_{\nu\in\lambda^{-3}\mathcal O_K} \frac{\tau(\nu) \overline{ \tau(1+d\nu)}}{\big(2|\nu|\cdot|1+d\nu|\big)^{-w} }\\
 &
 \hskip53pt
 \cdot \int\limits_0^1 \left(t^{-\frac43-w} +t^{-\frac23-w}\right) \left(\big(a_d(\nu)+1\big)\cdot t+\tfrac{ (1-t)^2}{2}  \right)^{w} dt\,dw. \end{align*}
Consequently
 \begin{equation}
 \label{SdsR1and2R2}
 S_d(s) = \mathcal R_1(s) +2\mathcal R_2(s) + \mathcal I_2(s).
 \end{equation}

 Now, for $\text{\rm Re}(w) = -2-3\varepsilon$ the product of Gamma functions $$\Gamma(s+\frac12+w) \Gamma(s+w)$$ has no poles for
 $4\varepsilon<\text{\rm Re}(s)<\frac12+2\varepsilon$. This shows that $\mathcal I_2(s)$ is a holomorphic function in the region $4\varepsilon<\text{\rm Re}(s)<\frac12+2\varepsilon$. It then follows (as in STEP 2 above) that $\mathcal I_2(s)$ satisfies the bound
\begin{equation}
\label{I2(s)Bound}
\mathcal I_2(s) \ll_\varepsilon  |s|^{\frac52+7\varepsilon}  e^{-\pi|s|}. 
\end{equation}

 Now, by Theorem \ref{SpectralSideBound} and Theorem \ref{Sd(s)Theorem}, the function
 $S_d(s)$ is meromorphic in the region $\text{\rm Re}(s) > \varepsilon.$
 Since we also know that $\mathcal I_2(s)$ is holomorphic in this region  the proof of  Proposition  \ref{R1and2R2prop} immediately follows from (\ref{SdsR1and2R2}) and (\ref{I2(s)Bound}).  As in the previous proof of Proposition \ref{R1plu2R2}, we already know that $S_d(s)$ has meromorphic continuation to $\text{\rm Re}(s) > 0$ by Theorems \ref{SpectralSideBound} and \ref{Sd(s)Theorem}.

\end{proof}

\vskip 10pt 
 \noindent
 $\underline{\text{{\bf STEP 5}}}$ {\bf : Meromorphic continuation of $L_d^\#(s)$ to Re$(s)>\frac13+\varepsilon$.}
\vskip 5pt
Recall the Picard hypergeometric function
$$\mathcal F(s,x) = \int\limits_0^1 \left( t^{s-\frac43} + t^{s-\frac23}  \right) \left(x\cdot t + \tfrac{(t-1)^2}{2}\right)^{-s}  dt, \quad\qquad (s\in\mathbb C,\; x>0),$$
where the above integral converges absolutely for $\text{\rm Re}(s)>\tfrac13.$
For $\text{\rm Re}(s)>1$, we also recall the  L-function
\begin{align*}
L_d^\#(s) & = \sum_{\nu\in\lambda^{-3}\mathcal O_K} \frac{\tau(\nu) \overline{\tau(1+d\nu)}}{\big(|\nu|\cdot|1+d\nu| \big)^s} \cdot\mathcal F\left(s, \tfrac{(d+1)^2}{2d}\right)\\
&
\hskip 30pt
 - \; s \sum_{\nu\in\lambda^{-3}\mathcal O_K} \frac{\tau(\nu) \overline{\tau(1+d\nu)}}{\big(|\nu|\cdot|1+d\nu| \big)^s} \left(a_d(\nu)-\tfrac{d^2+1}{2d}\right)\cdot \mathcal F\left(s+1, \tfrac{(d+1)^2}{2d}\right)
\end{align*}
where $$a_d(\nu) - \tfrac{d^2+1}{2d} \; \ll_d \; |\nu(1+dv)|^{-\frac12}.$$
We shall now show that  the function $L_d^\#(s)$ appears in the first two terms in a certain binomial expansion of the the residue function $\mathcal R_1(s)$ which can be written as
$$\mathcal R_1(s) =G_1(s) \sum_{\nu \in \lambda^{-3}\mathcal O_K} \tau(\nu) \overline{\tau(1 + d\nu)} |\nu (1 + d\nu)|^{-s} \cdot \mathcal F\big(s, a_d(\nu)+1\big),$$
where $$G_1(s) =  \frac{2^{-7s - 1}}{\pi^{2s - 1}}\frac{\Gamma(2s)^2}{\Gamma\left(s + \frac{1}{2}\right)^2}.$$
For $\text{\rm Re}(s)>\frac13$, the binomial expansion of the function $\mathcal F(s,x)$ around $x=x_0$ is given by
\begin{align*}
\mathcal F(s,x) & = \sum_{k=0}^\infty \left( \begin{matrix}-s\\k\end{matrix}  \right) \big(x-x_0\big)^k\int\limits_0^1 \left(t^{s - \frac{4}{3} + k} + t^{s - \frac{2}{3} + k}\right) \left(x_0\cdot  t + \tfrac{(t - 1)^2}{2}\right)^{-s - k} dt\\
& 
= \sum_{k=0}^\infty \left( \begin{matrix}-s\\k\end{matrix}  \right) \big(x-x_0\big)^k \mathcal F(s+k, x_0).
\end{align*}
It follows that by expanding $\mathcal F\big(s, a_d(\nu)+1\big)$ around $a_d(\nu)+1 = \tfrac{(d+1)^2}{2d}$ we obtain
\begin{align*}
&\mathcal R_1(s) = G_1(s) \Bigg[L_d^\#(s)  + \sum_{k=2}^\infty 
\hskip-2pt\left( \begin{matrix}-s\\k\end{matrix}  \right) \sum_{\nu\in\lambda^{-3}\mathcal O_K} \frac{\tau(\nu) \overline{\tau(1+d\nu)}}{\big(|\nu|\cdot|1+d\nu| \big)^s}
\left(a_d(\nu)-\tfrac{d^2+1}{2d}\right)^k\\
&
\hskip 278pt
\cdot\mathcal F\left(s+k, \tfrac{(d+1)^2}{2d}\right)
\Bigg].
\end{align*}

In a similar manner we see that $L_d^\#(s+\frac12)$ appears in the first two terms in a  binomial expansion of the the residue function $\mathcal R_2(s)$ which can be written as
$$\mathcal R_2(s) =G_2(s) \sum_{\nu \in \lambda^{-3}\mathcal O_K} \frac{\tau(\nu) \overline{\tau(1 + d\nu)}}{ |\nu (1 + d\nu)|^{s+\frac12}} \cdot \mathcal F\big(s+\tfrac12, \, a_d(\nu)+1\big),$$
where $$G_2(s) =  \frac{-2^{-7s-\frac12}}{\pi^{2s-1}}\frac{\Gamma(2s)^2}{\Gamma(s)^2}.$$
It again follows that by expanding $\mathcal F\big(s+\tfrac12,\, a_d(\nu)+1\big)$ around $a_d(\nu)+1 = \tfrac{(d+1)^2}{2d}$ we obtain
\begin{align*}
&\mathcal R_2(s) = G_2(s) \Bigg[L_d^\#(s+\tfrac12) + \sum_{k=2}^\infty 
\left( \begin{matrix}-s-\tfrac12\\k\end{matrix}  \right) \sum_{\nu\in\lambda^{-3}\mathcal O_K} \frac{\tau(\nu) \overline{\tau(1+d\nu)}}{\big(|\nu|\cdot|1+d\nu| \big)^{s+\frac12}}
\\ &\hskip 170pt \cdot \left(a_d(\nu)-\tfrac{d^2+1}{2d}\right)^k \mathcal F\left(s+\tfrac12+k, \tfrac{(d+1)^2}{2d}\right)
\Bigg].
\end{align*}
Since we have already proved in Proposition \ref{R1and2R2prop} that $\mathcal R_1(s) + 2\mathcal R_2(s)$ has a meromorphic continuation to $\text{\rm Re}(s) > 0$ and the sum 
$$\sum_{\nu\in\lambda^{-3}\mathcal O_K} \frac{\tau(\nu) \overline{\tau(1+d\nu)}}{\big(|\nu|\cdot|1+d\nu| \big)^s}
\left(a_d(\nu)-\tfrac{d^2+1}{2d}\right)^k$$
is absolutely convergent for $\text{\rm Re}(s) > 0$ and $k\ge 2$ it then follows from the previous computations that  $L_d^\#(s)$ has a meromorphic continuation to $\text{\rm Re}(s) > \frac13$.

 The sum (over $k\ge 2$) of the binomial terms $\left| \left(\begin{smallmatrix}-s\\k\end{smallmatrix}  \right)\right|$ grows exponentially in $|s|$, which implies that  the function
$$\sum_{k=2}^\infty 
\left( \begin{matrix}-s\\k\end{matrix}  \right) \sum_{\nu\in\lambda^{-3}\mathcal O_K} \frac{\tau(\nu) \overline{\tau(1+d\nu)}}{\big(|\nu|\cdot|1+d\nu| \big)^s}
\left(a_d(\nu)-\tfrac{d^2+1}{2d}\right)^k$$
has  at most exponential growth in $|s|$ for $\text{\rm Re}(s) > 0$ as $|s|\to\infty.$ In this manner it is possible to prove that $L_d^\#(s)$ has at most exponential growth (away from poles) in the region $\text{\rm Re}(s) >\frac13.$ For applications, however, we would like to do better and show it has polynomial growth instead of exponential growth. Obtaining polynomial growth can be achieved by breaking the sum over $\nu \in \lambda^{-3}\mathcal O_K$ into two sums, one sum over
$|\nu| \le |s|$ and the other over $|\nu| >|s|.$ This approach will be worked out in the next STEP.

\vskip 10pt 
 \noindent
 $\underline{\text{{\bf STEP 6}}}$ {\bf : Obtaining a polynomial bound in $|s|$ for $L_d^\#(s)$ away from poles.}
\vskip 5pt
 As explained at the end of STEP 5, in order to obtain a polynomial bound for $L_d^\#(s)$ it is necessary to break the $\nu$-sum in the definition of $L_d^{\#}(s)$ into two sums. This approach leads to the following definition.
  \begin{definition} {\bf (The function $ L_d^{\#\prime}(s)$)}
  Let $\text{\rm Re}(s) > 1$. Then we define
  \begin{align*}
    &L_d^{\#\prime}(s) := \sum_{\nu \in \lambda^{-3}\mathcal O_K, \;|\nu| > |s|} \tau(\nu) \overline{\tau(1 + d\nu)} |\nu (1 + d\nu)|^{-s} \\ &\hskip 150pt \cdot\int\limits_0^1 \left(t^{s - \frac{4}{3}} + t^{s - \frac{2}{3}}\right) \left(\tfrac{(d + 1)^2}{2d}t + \tfrac{(t - 1)^2}{2}\right)^{-s} dt \\ &\hskip 70pt - s \hskip-8pt\sum_{\nu \in \lambda^{-3}\mathcal O_K,\;|\nu| > |s|} \tau(\nu) \overline{\tau(1 + d\nu)} |\nu (1 + d\nu)|^{-s} \left(a_d(\nu) - \tfrac{d^2 + 1}{2d}\right) \\ &\hskip 150pt \cdot \int\limits_0^1 \left(t^{s - \frac{1}{3}} + t^{s + \frac{1}{3}}\right) \left(\tfrac{(d + 1)^2}{2d}t + \tfrac{(t - 1)^2}{2}\right)^{-s - 1} dt.
  \end{align*}
  \end{definition} 
 
 \begin{remark} Note that this is the same as the definition of $L_d^{\#}(s)$ but with the sums over $\nu$ restricted to $\nu \in \lambda^{-3}\mathcal O_K$ with $|\nu| > |s|$.
 \end{remark}

\begin{proposition}
\label{LsharpPrime} Fix $\varepsilon > 0.$ The function $L_d^{\#}(s) - L_d^{\#\prime}(s)$ has  holomorphic continuation to 
$\frac13+\varepsilon< \text{\rm Re}(s) <1$ and in this region satisfies the bound 
$
  \left|L_d^{\#}(s) - L_d^{\#\prime}(s)\right| \ll_\varepsilon |s|^3.
$
\end{proposition}
\begin{proof}
For $\frac13+\varepsilon< \text{\rm Re}(s) <1$ we have
\begin{align*}
  &L_d^{\#}(s) - L_d^{\#\prime}(s)  \;=  \underset{|\nu| \leq |s|}{\sum_{\nu \in \lambda^{-3}\mathcal O_K,} }
  \tau(\nu) \overline{\tau(1 + d\nu)} |\nu (1 + d\nu)|^{-s} 
  \\ &
  \hskip 170pt 
  \cdot \int\limits_0^1 \left(t^{s - \frac{4}{3}} + t^{s - \frac{2}{3}}\right) \left(\tfrac{(d + 1)^2}{2d}t + \tfrac{(t - 1)^2}{2}\right)^{-s} dt 
  \\ &
  \\
  &
  \hskip 50pt 
  - s\hskip-5pt\underset{|\nu| \leq |s|}{\sum_{\nu \in \lambda^{-3}\mathcal O_K,} } \tau(\nu) \overline{\tau(1 + d\nu)} |\nu (1 + d\nu)|^{-s} \left(a_d(\nu) - \tfrac{d^2 + 1}{2d}\right) 
  \\ &
  \hskip 160pt 
  \cdot \int\limits_0^1 \left(t^{s - \frac{1}{3}} + t^{s + \frac{1}{3}}\right) \left(\tfrac{(d + 1)^2}{2d}t + \tfrac{(t - 1)^2}{2}\right)^{-s - 1} dt.
\end{align*}

  Since $\tfrac{(d+a)^2}{2d}>1$ and $\frac13< \text{\rm Re}(s)$ both of the integrals in each summand above are bounded. Furthermore, the number of terms in each sum is less than a constant times $|s|^2$ and each term  $|\tau(\nu) \overline{\tau(1 + d\nu)}| \ll |s|$ for $\nu\ll |s|.$
 Therefore,
$
  \left|L_d^{\#}(s) - L_d^{\#\prime}(s)\right| \ll |s|^3.
$
 It is also clear  that $L_d^{\#}(s) - L_d^{\#\prime}(s)$  is holomorphic for $\text{\rm Re}(s) > \frac13$ because it is expressed as a finite sum in which each summand is holomorphic.
\end{proof}

\begin{proposition} \label{R1BoundProp} Fix $\varepsilon > 0.$ Then for $\frac13+\varepsilon < \text{\rm Re}(s) <1$ we have
\begin{align*}
 & \mathcal R_1(s)  = \frac{2^{-7s - 1}}{\pi^{2s - 1}}\frac{\Gamma(2s)^2}{\Gamma\left(s + \frac{1}{2}\right)^2} \left(L_d^{\#}(s) + \mathcal O\left(|s|^2 \right)\right),\\
  &
 \mathcal R_2(s)  = \frac{-2^{-7s-\frac12}}{\pi^{2s-1}}\frac{\Gamma(2s)^2}{\Gamma(s)^2} \left(L_d^{\#}\left(s+\tfrac12\right) + \mathcal O\left(|s|^2 \right)\right).
\end{align*}
\end{proposition}
\begin{proof}
Recall that
\begin{align*}
  \mathcal R_1(s) = \frac{2^{-7s - 1}}{\pi^{2s -1}}\frac{\Gamma(2s)^2}{\Gamma\left(s + \frac{1}{2}\right)^2} \sum_{\nu \in \lambda^{-3}\mathcal O_K} \tau(\nu) \overline{\tau(1 + d\nu)} |\nu (1 + d\nu)|^{-s} \\ \cdot  \int\limits_0^1 \left(t^{s - \frac{4}{3}} + t^{s - \frac{2}{3}}\right) \left((a_d(\nu) + 1) t + \tfrac{(1 - t)^2}{2}\right)^{-s} dt.
\end{align*}

We split the sum over $\nu \in \lambda^{-3}\mathcal O_K$ into the cases in which $|\nu| \leq |s|$ and $|\nu| > |s|$, i.e.
\begin{align*}
&
  \sum_{\nu \in \lambda^{-3}\mathcal O_K} \tau(\nu) \overline{\tau(1 + d\nu)} |\nu (1 + d\nu)|^{-s} \\ &\hskip 80pt \cdot \int\limits_0^1 \left(t^{s - \frac{4}{3}} + t^{s - \frac{2}{3}}\right) \left((a_d(\nu) + 1) t + \tfrac{(1 - t)^2}{2}\right)^{-s} dt \\
  &
  \hskip 15pt
  = \sum_{\nu \in \lambda^{-3}\mathcal O_K, \; |\nu| \leq |s|} \hskip-10pt\tau(\nu) \overline{\tau(1 + d\nu)} |\nu (1 + d\nu)|^{-s} \\ &\hskip 80pt \cdot \int\limits_0^1 \left(t^{s - \frac{4}{3}} + t^{s - \frac{2}{3}}\right) \left((a_d(\nu) + 1) t + \tfrac{(1 - t)^2}{2}\right)^{-s} dt 
  \\
  &
  \hskip 35pt + \hskip-10pt \sum_{\nu \in \lambda^{-3}\mathcal O_K, \; |\nu| > |s|} \hskip-10pt\tau(\nu) \overline{\tau(1 + d\nu)} |\nu (1 + d\nu)|^{-s} \\ &\hskip 80pt \cdot \int\limits_0^1 \left(t^{s - \frac{4}{3}} + t^{s - \frac{2}{3}}\right) \left((a_d(\nu) + 1) t + \tfrac{(1 - t)^2}{2}\right)^{-s} dt.
\end{align*}

For the $|\nu| > |s|$ case, we use the binomial theorem, yielding
\begin{align*}
  \sum_{k = 0}^{\infty} \binom{-s}{k} \sum_{\nu \in \lambda^{-3}\mathcal O_K, \;|\nu| > |s|} \tau(\nu) \overline{\tau(1 + d\nu)} |\nu (1 + d\nu)|^{-s} \left(a_d(\nu) - \tfrac{d^2 + 1}{2d}\right)^k \\ \cdot \int\limits_0^1 \left(t^{s - \frac{4}{3} + k} + t^{s - \frac{2}{3} + k}\right) \left(\tfrac{(d + 1)^2}{2d} t + \tfrac{(t - 1)^2}{2}\right)^{-s - k} dt.
\end{align*}
Note that the sum of the $k = 0$ and $k = 1$ terms is precisely the definition of $L_d^{\#\prime}(s)$, so that this series equals $L_d^{\#\prime}(s) + T_1(s)$ where
\begin{align*}
  T_1(s) :=\sum_{k = 2}^{\infty} \binom{-s}{k} \sum_{\nu \in \lambda^{-3}\mathcal O_K, \; |\nu| > |s|}\hskip-8pt \tau(\nu) \overline{\tau(1 + d\nu)} |\nu (1 + d\nu)|^{-s} \left(a_d(\nu) - \tfrac{d^2 + 1}{2d}\right)^k \\ \cdot \int\limits_0^1 \left(t^{s - \frac{4}{3} + k} + t^{s - \frac{2}{3} + k}\right) \left(\tfrac{(d + 1)^2}{2d} t + \tfrac{(t - 1)^2}{2}\right)^{-s - k} dt.
\end{align*}

We want to bound $T_1(s)$. Consider the integral
\[
  \int\limits_0^1 \left(t^{s - \frac{4}{3} + k} + t^{s - \frac{2}{3} + k}\right) \left(\tfrac{(d + 1)^2}{2d}t + \tfrac{(t - 1)^2}{2}\right)^{-s - k} dt.
\]
In particular, we claim that its absolute value is less than a constant times $c_d^{-k}$, where $c_d > 1$ is a constant that may depend on $d$. If $\frac{d}{d^2 + 1} \leq t \leq 1$, then $\tfrac{(d + 1)^2}{2d}t + \frac{(t - 1)^2}{2} \geq c_d^{\prime}$, where $c_d^{\prime} > 1$. For all $0 \leq t \leq 1$, $\frac{(d + 1)^2}{2d}t + \frac{(t - 1)^2}{2} \geq \frac{1}{2}$. Because $d \geq 2$, $0 < \frac{d}{d^2 + 1} < \frac{1}{2}$. Let $c_d = \min\left(\frac{d^2 + 1}{2d},c_d^{\prime}\right)$.

Additionally, $\left|a_d(\nu) - \tfrac{d^2 + 1}{2d}\right| \ll |\nu|^{-1} < |s|^{-1}$, and
\[
  \sum_{\nu \in \lambda^{-3}\mathcal O_K, \; |\nu| > |s|} \tau(\nu) \overline{\tau(1 + d\nu)} |\nu (1 + d\nu)|^{-s} \left(a_d(\nu) - \tfrac{d^2 + 1}{2d}\right)^2
\]
is bounded for $\text{\rm Re}(s) > 0$.

It follows that
\begin{align*}
   T_1(s)
 & \ll \;  \sum_{k = 2}^{\infty} \binom{|s| + k - 1}{k} (c_d|s|)^{-k + 2}
 \\
 &
 \ll \; \sum_{k = 0}^{\infty} \binom{|s| + k - 1}{k} (c_d|s|)^{-k + 2} = c_d^2|s|^2 \left(1 - c_d^{-1}|s|^{-1}\right)^{-|s|}.
\end{align*}

For the $|\nu| \leq |s|$ case, we first note that there are at most $\ll |s|^2$ terms in the sum
\begin{align*} &T_2(s) := \hskip-6pt\sum_{\nu \in \lambda^{-3}\mathcal O_K, \;|\nu| \leq |s|}
   \hskip-10pt
  \tau(\nu) \overline{\tau(1 + d\nu)} |\nu (1 + d\nu)|^{-s} \\ &\hskip 130pt \cdot \int\limits_0^1 \left(t^{s - \frac{4}{3}} + t^{s - \frac{2}{3}}\right) \left((a_d(\nu) + 1) t + \tfrac{(1 - t)^2}{2}\right)^{-s} dt.
\end{align*}
For each of those terms, $|\nu (1 + d\nu)|^{-s}$ and
\[
  \int\limits_0^1 \left(t^{s - \frac{4}{3}} + t^{s - \frac{2}{3}}\right) \left((a_d(\nu) + 1) t + \tfrac{(1 - t)^2}{2}\right)^{-s} dt
\]
are bounded. The number of summands in this case is less than a constant times $|s|^2$, so the sum is thus less than a constant times $|s|^2$.

Therefore, for $s\in\mathbb C$ with $\varepsilon<\text{\rm Re}(s)<1$ we have
\[
  \mathcal R_1(s) = \frac{2^{-7s - 1}}{\pi^{2s -1}}\frac{\Gamma(2s)^2}{\Gamma\left(s + \frac{1}{2}\right)^2} \left(L_d^{\#\prime}(s) + H_1(s)\right),
\]
where $H_1(s) := T_1(s) + T_2(s)$ and $|H_1(s)| \ll |s|^2$. An analogous argument tells us that
\[
  \mathcal R_2(s)  = -\frac{2^{-7s-\frac12}}{\pi^{2s-1}}\frac{\Gamma(2s)^2}{\Gamma(s)^2} \left(L_d^{\#\prime}\left(s+\tfrac12\right) + H_2(s)\right),
\]
where $|H_2(s)| \ll |s|^2$. Combining this with Proposition \ref{LsharpPrime} completes the proof  Proposition \ref{R1BoundProp}.
\end{proof}

\pagebreak
\vskip 10pt 
 \noindent
 $\underline{\text{{\bf STEP 7}}}$ {\bf : Completion of the proof of the first two parts of Theorem \ref{RelatingSd(s)}.}
\vskip 5pt
The first part of Theorem \ref{RelatingSd(s)}, given by
\begin{align}\label{FirstPart}
  &S_d(s) \,=\,  \frac{2^{-7s-1}}{\pi^{2s-1}}\frac{\Gamma(2s)^2}{\Gamma(s+\frac12)^2}\cdot L_d^\#(s)\; 
  -\;
  \frac{2^{-7s+\frac12}}{\pi^{2s-1}}\frac{\Gamma(2s)^2} {\Gamma(s)^2}\cdot  L_d^{\#}(s+\tfrac12) \; \\ \nonumber &\hskip 250pt + \;\mathcal O\bigg(|s|^{2\text{\rm Re}(s)+3} e^{-\pi|s|}   \bigg)
  \end{align}
for $\text{\rm Re}(s) > \frac13$, follows from Propositions \ref{R1plu2R2} and \ref{R1BoundProp}.
For the second part,  Theorems \ref{Theorem:SpectralBound} and \ref{Sd(s)Theorem} imply that 
\begin{align*}
 |S_d(s)| & \ll_\varepsilon \, |s|^{\textup{max}\left(2\text{\rm Re}(s) + \frac{5}{6},\;\frac{4}{3}\right) + \varepsilon}  e^{-\pi|s|} \;+ \; \left|\frac{\Gamma(2s) \Gamma\left(2s - \tfrac23\right)}{\Gamma\left(2s + \tfrac16\right)}\right|\\
& \ll_\varepsilon \, |s|^{\textup{max}\left(2\text{\rm Re}(s) + \frac{5}{6},\;\frac{4}{3}\right) + \varepsilon}  e^{-\pi|s|}
\end{align*}
for $\varepsilon < \text{\rm Re}(s)<1$ provided $|s-\rho|>\varepsilon$  for any pole $\rho\in\mathbb C$ of $S_d(s)$. The proof of the second part  now follows from the above bound for $|S_d(S)|$ together with the first part (\ref{FirstPart})
and the asymptotic formulae
\[
  \left|\frac{\Gamma(2s)^2}{\Gamma(s)^2}\right| \sim |s|^{2\text{\rm Re}(s)}e^{-\pi|s|}, \qquad \left|\frac{\Gamma(2s)^2}{\Gamma\left(s + \frac{1}{2}\right) \Gamma(s)}\right| \sim |s|^{2\text{\rm Re}(s) - \frac12}e^{-\pi|s|}.
\]

\vskip 10pt
$\underline{\text{{\bf STEP 8}}}${\bf : Determination of the poles of $L_d^{\#}(s)$ for $\text{\rm \bf Re}(s)>\frac13$.}

\vskip 5pt
The formula found in the first part of Theorem \ref{RelatingSd(s)} tells us that for each pole of $S_d(s)$ with $\textup{Re}(s) > \frac13$, either $L_d^{\#}(s)$ or $L_d^{\#}\left(s + \frac{1}{2}\right)$ has a pole at the same location, with the corresponding order. We also know that the poles of $L_d^{\#}(s)$ must have $\textup{Re}(s) \leq 1$. Thus if $S_d(s)$ has a simple pole at $s = \frac{2}{3}$, then $L_d^{\#}(s)$ has a simple pole at $s = \frac{2}{3}$.  Also, $L_d^{\#}(s)$ has a double pole at $s = \frac{1}{2}$ and simple poles at $s = \frac{1}{2} \pm it_j$, $t_j \neq 0$, $\left\langle u_j,\theta\overline{\theta_d} \right\rangle \neq 0$. If $L_d^{\#}(s)$ instead had any of the poles at $s = 1$ or $s = 1 \pm t_j$, then we would be forced either to have $S_d(s)$ have a pole at $s = 1$ or $s = 1 \pm it_j$ or to have $L_d^{\#}(s)$ have a pole at $s = \frac{3}{2}$ or $s = \frac{3}{2} \pm it_j$, both of which we know cannot happen. The computation of the residue of $L_d^{\#}(s)$ at each of its poles is then an immediate consequence of the formula.

\end{proof}

\section{Relating $L_d^{\#}(s)$ to $L_d(s)$}

We now extend the properties of $L_d^{\#}(s)$ given in Theorem  \ref{RelatingSd(s)}  to yield the following result about $L_d(s)$.

\begin{theorem}
The function $L_d(s)$ has meromorphic continuation to the region $\textup{Re}(s) > \frac{1}{2}$ with at most a simple pole at $s = \frac{2}{3}$ which occurs if and only if the Eisenstein contribution $\mathcal E(s)$  given in Theorem \ref{SpectralSide} has a pole at $s=\frac23.$ 

Fix $\varepsilon > 0.$ Then $L_d(s)$ has possible poles at the zeros of $\mathcal F\left(s,\frac{(d + 1)^2}{2d}\right)$ with $\frac{1}{2} < \textup{Re}(s) \leq 1$.
 In the region $\textup{Re}(s) > \frac{1}{2} + \varepsilon$  and $|s - \rho| > \varepsilon$ (for any pole $\rho$ of $L_d(s)$)  we have the bound
    $
      L_d(s) \ll_{d,\varepsilon} |s|^{\frac{7}{2}}.
   $
\end{theorem}

\begin{proof}
We first prove the following lemma, which we then use to show that the theorem is a consequence of the definition of $L_d^{\#}(s)$ and its properties that were found earlier.

\begin{lemma} \label{SteepestDescent}
  Fix $a > 0$, and let $s \in \mathbb C$ with $\textup{Re}(s) > \frac{1}{3}$.   Then if $s = \sigma + it$ and $\sigma$ is fixed with $|t| \rightarrow \infty$, we have
  \[
   \mathcal F(\sigma + it,a) = \left(\frac{2\pi}{|t|}\right)^{\frac{1}{2}} e^{-i (\log a)\cdot t} \left(2a^{-\sigma} + \mathcal O\left(|t|^{-1}\right)\right) a^{-\frac{1}{2}}.
  \]
\end{lemma}

\begin{proof}
 We may assume $t\to \infty.$ We use the method of steepest descent.
 Recall that
  \[
    \mathcal F(s,a) = \int\limits_0^1 \left(u^{- \frac{4}{3}} + u^{- \frac{2}{3}}\right) \left(a + \frac{(u - 1)^2}{2u}\right)^{-s} du.
  \]
  which can be written as  \[
    \mathcal F(\sigma + it, a) = \int\limits_0^1 f_{a,\sigma}(u)\, e^{itg_a(u)}\, du,
  \]
  where
  \[
    f_{a,\sigma}(u) = \left(u^{-\frac{4}{3}} + u^{-\frac{2}{3}}\right) \left(a + \frac{(u - 1)^2}{2u}\right)^{-\sigma}
  \]
  and
  \[
    g_a(u) = -\textup{log}\left(a + \frac{(u - 1)^2}{2u}\right).
  \]
  We then compute
  \[
    g_a^{\prime}(u) = \frac{-u^2 + 1}{u \left(u^2 + 2 (a - 1) u + 1\right)}
  \]
  and
  \[
    g_a^{\prime\prime}(u) = \frac{u^4 - 4u^2 - 4 (a - 1) u - 1}{u^2 \left(u^2 + 2 (a - 1) u + 1\right)^2}.
  \]
  Thus there is a non-degenerate saddle point at $u = 1$, and if $a \neq 2$, there is a non-degenerate saddle point at $u = -1$. These are the only saddle points.

  May may now apply the saddle point method; the relevant statement for our case is as follows (see \cite{MR950167}).

  \begin{theorem} \label{SaddlePointMethod}
   Suppose that $f(z)$ and $S(z)$ are holomorphic functions on an open, bounded, and simply connected set $\Omega_x \in \mathbb C^n$ such that $I_x = \Omega_x \cap \mathbb R^n$ is simply connected, $\textup{Re}(S(z))$ has a single maximum $x^0$ in $I_x$, and $x^0$ is a non-degenerate saddle point of $S$. Then as $\lambda \rightarrow \infty$, $I(\lambda) = \int\limits_{I_x} f(x) e^{\lambda S(x)} dx$ satisfies the asymptotic formula
    \[
      I(\lambda) = \left(\frac{2\pi}{\lambda}\right)^{\frac{n}{2}} e^{\lambda S\left(x^0\right)} \Big(f\left(x^0\right) + \mathcal O\left(\lambda^{-1}\right)\Big) \prod_{j = 1}^n (-\mu_j)^{-\frac{1}{2}}
    \]
    where $\mu_j$ are the eigenvalues of the Hessian of $S$ and $(-\mu_j)^{-\frac{1}{2}}$ are chosen so that they satisfy $\left|\textup{arg}\sqrt{-\mu_j}\right| < \frac{\pi}{4}$.
  \end{theorem}

  It follows from Theorem \ref{SaddlePointMethod} and the previous computations that for $\sigma$ fixed and $t \rightarrow \infty$, we have
  \[
   \mathcal  F(\sigma + it, a) = \left(\frac{2\pi}{t}\right)^{\frac{1}{2}} e^{-i \left(\textup{log}a\right) t} \Big(2a^{-\sigma} + \mathcal O\left(t^{-1}\right)\Big) a^{-\frac{1}{2}}.
  \]
\end{proof}

Recall that
  \[
    L_d^{\#}(s) = \mathcal F\left(s,\tfrac{(d + 1)^2}{2d}\right) L_d(s)\; - \;s\cdot \mathcal F\left(s + 1,\tfrac{(d + 1)^2}{2d}\right) L_d^*(s).
  \]
  It was shown in Theorem  \ref{RelatingSd(s)} that the only possible pole of $L_d^{\#}(s)$ with $\text{Re}(s) > \frac{1}{2}$ is a possible simple pole at $s = \frac{2}{3}$ and that  $L_d(s)$ is holomorphic for $\textup{Re}(s) > 1$, $L_d^*(s) \approx L_d\left(s + \frac{1}{2}\right)$ for $\textup{Re}(s) > \frac{1}{2}$, and $\mathcal F\left(s,\frac{(d + 1)^2}{2d}\right)$ is holomorphic for $\textup{Re}(s) > \frac{1}{3}$. Thus if $L_d^{\#}(s)$ has a simple pole at $s = \frac{2}{3}$, then $L_d(s)$ must have a simple pole at $s = \frac{2}{3}$, and any other poles of $L_d(s)$ with $\textup{Re}(s) > \frac{1}{2}$ must be at zeros of $\mathcal F\left(s,\frac{(d + 1)^2}{2d}\right)$, with the order of the pole, if it occurs, being less than or equal to the order of vanishing of $\mathcal F\left(s,\frac{(d + 1)^2}{2d}\right)$ at that point, since in this region $L_d^{\#}(s)$ does not have any other poles and the second summand on the right side is holomorphic.

Now, for $\textup{Re}(s) > \frac{1}{2} + \varepsilon$, we have $\left|s - \frac{2}{3}\right| > \varepsilon$, $L_d^{\#}(s) \ll_{\varepsilon} |s|^3$. Since the second summand in the definition of $L_d^{\#}(s)$ is asymptotically smaller than the first summand, on that region, we have the bound
  \[
   \mathcal F\left(s,\tfrac{(d + 1)^2}{2d}\right) L_d(s) \ll_{\varepsilon} |s|^3.
  \]
  It follows from Lemma \ref{SteepestDescent}
that as $t \rightarrow \infty$,
  \[
    \left|\mathcal F\left(\sigma + it,\frac{(d + 1)^2}{2d}\right)\right| \sim c|t|^{-\frac{1}{2}}
  \]
  for some constant $c > 0$. 
  This immediately implies that for $|s-\rho| >\varepsilon$ (at poles $\rho$ of $L_d(s)$)  we have the bound
  $|L_d(s)| \ll |s|^\frac72.$ 
\end{proof}

\subsection*{Acknowledgements}
Dorian Goldfeld is partially supported by Simons Collaboration Grant 567168.


\end{document}